\documentclass[a4paper,onesided]{amsart}

\usepackage[ansinew]{inputenc}
\usepackage{amsfonts}
\usepackage{amssymb, bbm, dsfont}
\usepackage{amsmath,amsthm,amsfonts,amssymb}
\usepackage[left=2.25cm,right=2.25cm,top=4cm,bottom=4cm]{geometry}
\usepackage{color}
\usepackage{enumerate}

 \theoremstyle{plain}
\newtheorem{lem}{Lemma}[section]

\newtheorem{thm}[lem]{Theorem}
\newtheorem{prop}[lem]{Proposition}
\theoremstyle{definition}
\newtheorem{Def}[lem]{Definition}
\newtheorem{Defs}[lem]{Definitions}
\newtheorem{assum}[lem]{Assumptions}
\newtheorem{Rem}[lem]{Remark}
\newcommand{\h}{\mathbf{H}}
\newcommand{\ve}{\mathbf{V}}

\newcommand{\pmat}{\begin{pmatrix}}
\newcommand{\epmat}{\end{pmatrix}}

\newcommand{\DEQS}{\begin{eqnarray*}}
\newcommand{\EEQS}{\end{eqnarray*}}
\newcommand{\DEQSZ}{\begin{eqnarray}}
\newcommand{\EEQSZ}{\end{eqnarray}}

\newcommand{\lqq}{\lefteqn}

\newcommand{\lk}{\left}
\newcommand{\rk}{\right}

\newcommand{\levy}{L\'evy }

\newcommand{\bu}{\mathbf{u}}
\newcommand{\bv}{\mathbf{v}}

\newcommand{\bun}{\mathbf{u}_n}

\newcommand{\bxi}[1]{\bx^{(#1)}}
\newcommand{\gi}[1]{\gamma^{(#1)}}
\newcommand{\biz}[1]{\mathcal{B}_z^{(#1)}}
\newcommand{\ji}[1]{J^{(#1)}}
\newcommand{\ki}[1]{K^{(#1)}}
\newcommand{\zl}[1]{Z^\lambda(#1)}
\newcommand{\ekl}{\tilde{\eta}^\lambda}
\newcommand{\ebl}{\bar{\eta}^\lambda}
\newcommand{\efl}{\nu^\lambda}

\newcommand\del[1]{}

\newcommand\cadlag{c{\`a}dl{\`a}g }
\newcommand{\lve}{\lvert}
\newcommand{\rve}{\rvert}
\newcommand{\Lve}{\lVert}
\newcommand{\Rve}{\rVert}

\newcommand{\err}{\mathbf{R}}
\newcommand{\bb}{\mathbf{B}}

\newcommand{\EE}{\mathbb{E}}
\newcommand{\CP}{{{ \mathcal P }}}

\newcommand{\rA}{\mathrm{A}}

\newcommand{\bx}{\mathbf{X}}
\DeclareMathOperator{\gal}{Linspan}

\def\eps{\varepsilon}
\newcommand{\cprn}{\CP^R_{t,n}}
\newcommand{\burn}{\bu^R_n}
\newcommand{\Burn}{U^R_n}

\title{Ergodicity of stochastic shell models driven by pure jump noise}

\author[H. Bessaih]{Hakima Bessaih}
\address[H. Bessaih]{Department of Mathematics, University of Wyoming, 1000 East University Avenue, Laramie WY 82071, United States,}
\author[E. Hausenblas]{Erika Hausenblas}
\author[P. Razafimandimby]{Paul Razafimandimby}
\address[E.~Hausenblas and P.~Razafimandimby]{Department of Mathematics and Information Technology, Montanuniversit\"at Leoben,
Fr. Josefstr. 18, 8700 Leoben, Austria}
\email[H.~Bessaih]{bessaih@uwyo.edu}
\email[E.~Hausenblas]{erika.hausenblas@unileoben.ac.at}
\email[P.~Razafimandimby]{paul.razafimandimby@unileoben.ac.at}
\begin{document}
\begin{abstract}
 In the present paper we study a stochastic evolution equation for shell (SABRA \& GOY) models with pure jump \levy noise $L=\sum_{k=1}^\infty l_k(t)e_k$ on a Hilbert space $\h$. Here $\{l_k; k\in \mathbb{N}\}$ is a family of independent and identically distributed (i.i.d.) real-valued pure jump \levy processes and $\{e_k; k\in \mathbb{N}\}$ is an orthonormal basis of $\h$.  We mainly prove that the stochastic system has a unique invariant measure. For this aim we show that  if the \levy measure of each component $l_k(t)$ of $L$ satisfies a certain order and a non-{degeneracy} condition and is  absolutely continuous with respect to the Lebesgue measure, then the
 Markov semigroup associated { with} the unique solution of the system has the strong Feller property. If, furthermore, each $l_k(t)$ satisfies a small deviation property, then 0 is accessible for the dynamics independently of the initial
 condition. Examples of noises satisfying our conditions are a family of i.i.d tempered \levy noises $\{l_k; k\in \mathbb{N}\}$ and $\{l_k=W_k\circ G_k + G_k; k\in \mathbb{N} \}$ where $\{G_k; k \in \mathbb{N}\}$ (resp., $\{W_k; k\in \mathbb{N}\}$) is a sequence of i.i.d subordinator Gamma (resp., real-valued Wiener) processes  with \levy density $f_G(z)=(\vartheta z)^{-1} e^{-\frac z\vartheta} \mathds{1}_{z>0}$. The proof of the strong Feller property relies on the truncation of the nonlinearity and the use of a gradient estimate for the Galerkin system of the truncated equation.
 The gradient estimate is a consequence of a Bismut-Elworthy-Li (BEL) type formula that we prove in the Appendix A of the paper.
\end{abstract}
\maketitle
\section{Introduction}

 In many applied sciences such as
aerodynamics, weather forecasting and hydrology,
 numerical investigation of three dimensional Navier-Stokes equations at high Reynolds' number is ubiquitous. Unfortunately,
  even with the most sophisticated scientific tools, it is a very
challenging task to compute analytically or via direct numerical
simulations the turbulent behavior of  3-D incompressible fluids.
This is due to the large range of scale of motions that need to be
resolved. To tackle this issue,  several models of turbulence that can capture the
physical phenomenon of turbulence in fluid flows at lower
computability cost have been proposed over the last three decades.
One class of these models of turbulence are the \textbf{shell
models}. There are various kind of shell models, but the most
popular in the physics and mathematics literature are the GOY and
SABRA models. The shell models basically describe the evolution of
complex Fourier-like components, denoted
by $u_{n}$ with the associated wave numbers denoted by $k_{n}$ where the discrete index $n$ is referred as the shell index, of a velocity field $u$.
The evolution of the infinite sequence $\{u_n\}_{n=-1}^\infty$ is
given by
\begin{equation}\label{SHELL-0}
\dot u_n(t)+{ \kappa} k_n^2 u_n(t)+b_n(u(t),u(t))=f_n(t, u(t)), \qquad
n=1,2,\ldots
\end{equation}
with $u_{-1}=u_0=0$ and $u_n(t) \in \mathbb C$ for $n \ge 1$. Here
${ \kappa} \ge 0$ and in analogy with Navier-Stokes equations ${ \kappa}$
represents a kinematic viscosity; $k_n=k_0 \lambda^n$
($\lambda>1$) and $f_n$ is a forcing term. The exact form of
$b_n(u,v)$ varies from one model to the other. However in all the
various models, it is assumed that $b_n(u,v)$ is chosen in such a
way that
\begin{equation}\label{incomp_shell}
 \Re \biggl(\displaystyle\sum_{n=1}^{\infty}b_n(u,v)\overline{v}_{n}\biggr)=0,
\end{equation}
where $\Re$ denotes the real part and $\overline x$
 the complex conjugate of $x$.

In particular, we define the bilinear terms $b_n$  as
\[b_n(u,v)=i(a k_{n+1}\overline u_{n+1} \overline v_{n+2}+
            b k_{n}\overline u_{n-1} \overline v_{n+1}-
            a k_{n-1}\overline u_{n-1} \overline v_{n-2}-
            b k_{n-1}\overline u_{n-2} \overline v_{n-1})
\]
in the GOY model (see \cite{G,goy}) and by
\[
b_n(u,v)=-i(a k_{n+1}\overline u_{n+1}  v_{n+2}+
            b k_{n}\overline u_{n-1} v_{n+1}+
            a k_{n-1} u_{n-1} v_{n-2}+
            b k_{n-1} u_{n-2} v_{n-1})
\]
in the SABRA model (see \cite{sabra}). Here $a, b$ are real
numbers. Note that equation \eqref{incomp_shell} implies a formal
law of conservation of energy in the inviscid and 
unforced form of \eqref{SHELL-0}.  The shell
models have similar properties to 2D fluids. In fact, they 
basically consist of infinitely many differential equations
having a structure similar to the Fourier representation of the
Navier-Stokes equations. They are constructed in such a way that
they and the Navier-Stokes equations at high Reynolds' number
exhibit similar statistical properties. Indeed both shell models
and Navier-Stokes models have a finite number of degrees of
freedom, see, for instance, \cite{Constantin}.
Another feature that shell models share with the Navier Stokes
equations is the so called {\sl determining modes}, see the
pioneering work of Foia{\c{s}} and Prodi \cite{MR0223716} for the
case of Navier-Stokes equations and \cite{Constantin} for the
shell models. Furthermore, the interactions in the Fourier space
for the shell models are local and therefore are easier to handle.
As such shell models are much simpler than the Navier-Stokes
equations and are more suitable for the analytical and numerical
investigation towards the understanding of turbulence. Due to
these facts these models and their stochastic counterparts have
been the subject of intensive numerical and analytical studies
during the last two decades. We refer, for instance, to the works
of Barbato et al \cite{Flandoli1}, Bessaih et al \cite{Flandoli3},
Constantin et al \cite{Constantin}, and Ditlvesen \cite{Ditlevsen}
 for more recent and detailed review of results related to the physical and mathematical theory of shell models.


In recent years the mathematical analysis of stochastic partial
differential equations driven by \levy processes began to draw
more and more attention. There are several examples where the
Gaussian noise is not well suited to represent realistically 
external forces. For example, if the ratio between the time scale
of the deterministic part and that of the stochastic noise is
large, then the temporal structure of the forcing in the course of
each event has no influence on the overall dynamics, and - at the
time scale of the deterministic process - the external forcing can
be modeled as a sequence of episodic instantaneous impulses. This
happens for example in Climatology (see, for instance,
\cite{MR2205118}). Often the noise observed by time series is
typically asymmetric, heavy-tailed and has  non trivial kurtosis.
These are all features which cannot be captured by a Gaussian
noise, but by a \levy noise with appropriate parameters. From the
mathematical point of view, \levy randomness requires other
techniques, and is  intricate and far from amenable to
mathematical analysis. Despite these facts the mathematical study
of the long-time behavior, in particular ergodicity, of SPDEs with
\levy noise are still at its infancy. This is mainly due to the
fact that the numerous results for Wiener driven models cannot be
in general transferred to SPDEs driven by L\'evy noise. The
analysis of the long-time behavior of SPDEs is more complicated
for L\'evy driven SPDEs. In addition, the dynamical behavior of
SPDEs changes essentially, if the Brownian noise is replaced by
L\'evy noise. E.g.\, Imkeller and Pavlyukevich investigated in
\cite{MR2220901,MR2205118} the dynamical behavior of systems
driven by a L\'evy noise and showed that the escape times from
certain potentials are exponentially distributed and differ
essentially from the escape times of the corresponding dynamical
systems driven by Brownian noise. One should note that there
are now several papers treating the ergodicity of nonlinear SPDEs
with \levy noise, see for example \cite{BG+CH}, \cite{SP-et
al}, \cite{Nersesyan},  \cite{SP+JZ}, \cite{Priola-et-al},
\cite{Priola et al-2} and \cite{EP+JZ}.

 In the present paper we investigate the ergodicity of stochastic shell models driven by random external forcing of jump type.
 More precisely, we are interested in a model equation of the form
\begin{equation}
 \begin{split}\label{SHELL-1}
  & d\bu(t)+[\kappa \rA\bu(t)+\bb(\bu(t),\bu(t))]dt= \sum_{k=1}^\infty \biggl(\int_{\mathbb{R}_0} z d\bar{\eta}_k(z,t)\biggr)\beta_k e_k
  \\
  & \bu(0)=\xi\in \h,
 \end{split}
 \end{equation}
where $\kappa$ is a positive number, $\{e_k, k=1, 2, \ldots\}$ is the orthonormal basis of a given Hilbert space $\h$ and
$$\int_{\mathbb{R}_0} z d\bar{\eta}_k(z,t):= \int_{\mathbb{R}_0} \mathds{1}_{\{\lvert z\rvert\le1\} } z\tilde{\eta}_k(dz,dt)
  +\int_{\mathbb{R}_0} \mathds{1}_{\{\lvert z\rvert>1\} } z \eta_k(dz,dt).$$
In \eqref{SHELL-1}, $\rA$ is a linear map  and $\bb$ is a bilinear
map on the underlying Hilbert space $\h$. The family $\{ \eta_k;
k=1, 2, \ldots \}$ represents a family of mutually independent Poisson random measures
with $\sigma$-finite \levy measures $\{\nu_k; k=1, 2, \ldots\}$ on
$\mathbb{R}_0:=\mathbb{R}\backslash \{0\}$. For each $k$ the
symbol $\tilde{\eta}_k$ represents the compensated
 Poisson random measure associated to $\eta_k$ and the family of compensators is denoted by $\{\nu_k(dz)dt; k=1, 2, \ldots\}$. The family $\{\beta_k; k=1, 2, \ldots\}$
is a family of positive numbers representing the roughness of the
noise. The maps $\rA$ and $\bb$ are carefully chosen so that it
can model the nonlinear terms of the GOY and SABRA shell models
defined { previously}.

In this paper, we mainly prove that if  the \levy measure of each component $l_k(t)$ of $L$ satisfies a certain order and non-{degeneracy} condition, and is { absolutely continuous} w.r.t. the Lebesgue measure and if each \levy process $l_k(t):=\int_{\mathbb{R}_0} z d\bar{\eta}_k(z,t)$ satisfies a small deviation property, then the stochastic evolution equations \eqref{SHELL-1} has a
unique invariant measure (see Theorem \ref{SHELL}). Examples of noises satisfying our conditions are a family of i.i.d tempered \levy noises $\{l_k; k\in \mathbb{N}\}$ and $\{l_k=W_k\circ G_k + G_k; k\in \mathbb{N} \}$ where $\{G_k; k \in \mathbb{N}\}$ (resp., $\{W_k; k\in \mathbb{N}\}$) is a sequence of i.i.d subordinator Gamma (resp., real-valued Wiener) processes  with \levy density $f_G(z)=(\vartheta z)^{-1} e^{-\frac z\vartheta} \mathds{1}_{z>0}$. { We mainly show that}
the Markovian semigroup associated with the solution of
\eqref{SHELL-1} has strong Feller property and that $0\in \h$ is an accessible point for the dynamic. The strong Feller property is the most 
challenging part of the proof. The
strategy of the proof of this result is based on the work
\cite{FF+BM}. Namely, we firstly truncate the nonlinearity \eqref{SHELL-1} and show  that the Galerkin approximation
of this modified/truncated version of \eqref{SHELL-1} has the strong
Feller property. This was achieved thanks to a Bismut-Elworthy-Li (BEL) type formula (see Lemma \ref{BEL}) that we state and prove in Appendix \ref{SEC-BEL}.
 Lemma
 \ref{BEL} is very similar to \cite[Theorem 1]{Takeuchi}. However, for each $n\in \mathbb{N}$ our Galerkin equations is a system of stochastic differential
 equations driven by random measures with L\'evy measure on $\mathbb{R}^n$ which, in contrast to \cite{Takeuchi}, do not necessarily have a smooth density. Nevertheless,
 we should note that
 the idea of the proof of Lemma \ref{BEL} is based on some modifications of \cite[Proof of Theorem 1]{Takeuchi}  and some arguments (change of measures)
 from \cite{Ishikawa}. The main assumptions for this BEL type formula to hold is an order condition type, a non-degeneracy (see Assumption \ref{Lev-Meas}-\eqref{LMiii}) and absolute continuity w.r.t. the Lebesgue measure (Assumption \ref{Lev-Meas}-\eqref{LMi}) of  the \levy measure of each $l_k(t)$. Secondly, we prove that the truncated equations itself has the strong Feller property. This result is based on \cadlag property of
 stochastic convolution $\mathfrak{S}(t)$ which is solution to the following equation
 $$ d\mathfrak{S}(t)+\kappa \rA\mathfrak{S}(t) dt= \sum_{k=1}^\infty \biggl(\int_{\mathbb{R}_0} z d\bar{\eta}_k(z,t)\biggr)\beta_k e_k
  , \quad
   \mathfrak{S}(0)=0.$$
Here recent results about time regularity of stochastic convolutions proved in \cite{SP+JZ-13} play an important role. Thirdly, since the solution to the original equation
 has good moment estimates we can show that the strong Feller property is preserved  when we remove the truncation
 function. To complete the proof of our main result we show in
 Proposition \ref{IRRED} that any ball centered at 0 with
 sufficiently large radius is visited, with positive probability, by the process $\bu$
 independently of the initial condition $\xi\in \h$. This fact
 holds under the condition that each one dimensional \levy process
 $l_k(t)$ has the
 small deviation property (see Proposition \ref{small-dev}).

 We should note that the nonlinear term of \eqref{SHELL-1} does not fall in the framework of the papers \cite{BG+CH, Nersesyan, Priola-et-al, Priola et al-2, EP+JZ}. 
 The paper \cite{ZD+YX} and the book \cite{KS12} studied the uniqueness of invariant measure
 associated to the Markov semigroup of the stochastic Navier-Stokes equations with \levy noise.  In \cite{ZD+YX} the driving noise is the sum of a non-degenerate
 Wiener noise and an infinite activity jump process. Thanks to the non-degeneracy of the Wiener noise the gradient estimate method in \cite{FF+BM, Fe97} can be adapted to their framework.
 In our case we closely follow the scheme in \cite{FF+BM} for the proof of the smoothing property of the semigroup, but we have to prove a BEL-type formula for pure jump noise.
 The authors of \cite{KS12} state that stochastic Navier-Stokes equations with Poisson process as a noise term has a unique invariant measure. They
 also provide the rate of convergence to the invariant measure. The proof of these results follows from the arguments in \cite{Nersesyan} which are very sophisticated and complicated to be explained here.
 Since the \levy measure of a Poisson process is a finite measure and we consider \levy processes with $\sigma$-finite \levy measure,
 the proof of \cite{Nersesyan} could not be used for
 our model.

 To close this introduction we give the structure of this paper. In Section \ref{Nota}
 we define most of the notations used in this paper 
 and  the assumptions frequently imposed
 throughout the paper. The main result (see Theorem \ref{Main
 result}) is stated and proved in Section \ref{ERG-SHELL}. The
 proof of this main theorem relies on two important propositions
 (Proposition \ref{TRU-CUT} and Proposition \ref{IRRED}) that are
 also stated and proved in Section \ref{ERG-SHELL}. Section
 \ref{SEC-MOD-SHELL} is devoted to the analytical study of the
 truncated version of \eqref{Hydro} (see Eq. \eqref{Hydro-1}).
 There, we mainly prove that the finite dimensional approximation of the
 truncated equations satisfy the strong Feller property which is
 preserved by passage to the limit. In Appendix \ref{SEC-BEL}, we
 prove a Bismut-Elworthy-Li type formula for system of SDEs driven
 by pure jump noise. In Appendix \ref{SEC-GRAD-EST} we prove an
 estimate for the gradient of the semigroup of the system of SDEs
 { from}  Appendix \ref{SEC-BEL}. In Appendix \ref{APP-CONV} we derive the
 necessary convergence which enables us to transfer the strong
 Feller property from the semigroup of the Galerkin approximation of the truncated
 equations to the semigroup of the infinite dimensional truncated
 equation.

\section{Notation and Assumptions}\label{Nota}
In this section we will introduce the necessary notation and
assumptions in this paper. We will mainly follow the notation in
\cite{HB+BF-12}.

Throughout this work we will identify the field of complex numbers
$\mathbb{C}$ with $\mathbb{R}^2$. That is, any complex number of
the form $x=x_1+ix_2$ will be identified with $(x_1,x_2)\in
\mathbb{R}^2$. As usual, for $x=(x_1,x_2) \in \mathbb R^2$ we set
$|x|^2= x_1^2+x_2^2$ and $x\cdot y=x_1y_1+x_2y_2$ is the scalar
product in $\mathbb R^2$. For a Banach space $\mathrm{B}$ we
denote by $\mathrm{B}^\ast$ its dual space.

Let $\h$ be the space defined by
$$
\h=\{\bu=(\bu_1, \bu_2, \ldots) \in (\mathbb R^2)^\infty:
 \sum_{n=1}^\infty|\bu_n|^2<\infty \}.
$$
This is a Hilbert space and we denote its norm by  $\vert \cdot \rvert$ and its scalar
product by $\langle \bu,\bv\rangle=\sum_{n=1}^\infty \bu_n
\cdot \bv_n$.

Let $\rA$ be a  linear map with domain
$D(\rA)$ on $\h$. We impose the following set of conditions on
$\rA$.
\begin{assum}\label{A}
\begin{enumerate}[(i)]
\item \label{Ai} We assume that $\rA$ is a self-adjoint positive
operator and its domain $D(\rA)$ is dense and
compact in $\h$.

 \item \label{Aii} We assume that there exists $k_0$ and $\lambda>1$
such that the eigenvalues of $\rA$ are given by
$$ \lambda_j=k_0 \lambda^{2j}.$$

\item \label{Aiii} We also suppose that the eigenfunctions $\{
e_1, e_2, \ldots \}$ of $\rA$ form an orthonormal basis of $\h$.

\end{enumerate}
\end{assum}
 With Assumption \ref{A}-\eqref{Ai} the fractional power
operators $\rA^\gamma$, $\gamma\ge 0$ are well-defined; they are
also self-adjoint, positive and invertible with inverse
$\rA^{-\gamma}$. We denote by $\ve_\gamma:=
D(\rA^\frac{\gamma}{2})$, $\gamma\ge 0$ the domain of
$\rA^\gamma$. It is a Hilbert space endowed with the graph norm. The
dual space $\ve_\gamma^\ast$ of $\ve_\gamma$, $\gamma\ge 0$, w.r.t.
to the inner product of $\h$ can be identified with
$D(\rA^{-\frac{\gamma}{2}})$. For $\gamma=\frac12$ we set $\ve:=
\ve_1$ and we denote its norm by $\Vert \cdot \Vert:= \vert \cdot \vert + \lvert \rA^\frac12 \cdot \vert $. Observe also that Assumption \ref{A} implies the following Poincar\'e type inequality
\begin{equation}\label{Poincare}
 \lve \cdot \rve^2\le \lambda_1^{-1}\Lve \cdot \Rve^2,
\end{equation}
where $\lambda_1$ is the smallest eigenvalue of $\rA$. Thus the
norm $\Vert \cdot \Vert$ is equivalent to $\lvert \rA^\frac12
\cdot\vert$.

When identifying $\h$ with its dual $\h^\ast$ we have the Gelfand
triple $\ve\subset \h \subset \ve^\ast$. We denote by $\langle
\bu,\bv\rangle$ the duality between $\ve^\ast$ and $\ve$ such that
$\langle \bu,\bv\rangle=(\bu,\bv)$ for $\bu\in \h$ and $\bv\in
\ve$.

Now, let $\bb:\h \times \h \to \ve^\ast$ be a bilinear
map satisfying the following set of conditions.
\begin{assum}\label{B}
 We assume that $\bb:\h \times \h \to \ve^\ast$ is a bilinear map
 satisfying the following three properties.
 \begin{enumerate}[(a)]
  \item \label{B-a} There exists a number $C_0>0$ such that for any $\bu, \bv\in \h$
  \begin{equation}
\Vert \bb(\bu,\bv)\Vert_{\ve^\ast}\le C_0 \lvert \bu \rvert \lvert
\bv\rvert.
  \end{equation}
  \item \label{B-b} There exists a constant $C_1>0$ such that
  \begin{equation}
\vert \bb(\bu,\bv)\rvert\le
\begin{cases}
C_1 \Vert \bu\Vert\, \lvert \bv \rvert \text{ for any } \bu \in
\ve
,\bv \in \h,\\
C_1 \lvert \bu \rvert  \,\Vert \bv\Vert \text{ for any } \bv \in
\ve
,\bu \in \h. \\
\end{cases}
  \end{equation}
  \item \label{B-c} For any $\bu \in \h, \bv \in \ve$
  \begin{equation}
\langle \bb(\bu,\bv),\bv\rangle=0.
  \end{equation}
 \end{enumerate}
\end{assum}

In our framework $\bb$ is defined by $\bb:(
{\mathbb
R}^2)^\infty\times ({\mathbb R}^2)^\infty\to ({\mathbb
R}^2)^\infty$
\[
 \bb(u,v)=(B_1(u,v), B_2(u,v),\ldots),
\]
where  $B_n=(B_{n,1},B_{n,2})$ and $B_{n,1}$ and $B_{n,2}$ are,
respectively,  the real parts and the imaginary parts of the $b_n$
given in the previous section. For instance, in the SABRA model
\begin{align}
\begin{split}
&B_{1,1}(u,v)=ak_{2} [-u_{2,2}  v_{3,1} +u_{2,1}  v_{3,2}]
\\
&B_{1,2}(u,v)=-a k_{2} u_{2} \cdot v_{3}
\end{split}
&\\
\begin{split}
&B_{2,1}(u,v)=ak_{3} [-u_{3,2}  v_{4,1} +u_{3,1}  v_{4,2}]
           +b k_{2} [-u_{1,2} v_{3,1}+u_{1,1} v_{3,2}]
\\
&B_{2,2}(u,v)=-a k_{3} u_{3}\cdot  v_{4} -b k_{2}  u_{1} \cdot
v_{3}
\end{split}
&\\
\intertext{and for $n >2$}
\begin{split}
&B_{n,1}(u,v)= \; ak_{n+1} [-u_{n+1,2}  v_{n+2,1} +u_{n+1,1}
v_{n+2,2}]
\\&\qquad\qquad\qquad   +b k_{n} [-u_{n-1,2} v_{n+1,1}+u_{n-1,1} v_{n+1,2}]
\\&\qquad\qquad\qquad           +a k_{n-1}[u_{n-1,2} v_{n-2,1}+u_{n-1,1} v_{n-2,2}]
\\&\qquad\qquad\qquad           +b k_{n-1}[u_{n-2,2} v_{n-1,1}+u_{n-2,1} v_{n-1,2}],
\end{split}
&\\
\begin{split}
&B_{n,2}(u,v)=-a k_{n+1} [u_{n+1,1}  v_{n+2,1} +u_{n+1,2}
v_{n+2,2}]
\\&\qquad\qquad\qquad    -b k_{n}  [u_{n-1,1} v_{n+1,1}+u_{n-1,2} v_{n+1,2}]
\\&\qquad\qquad\qquad    -a k_{n-1} [u_{n-1,1} v_{n-2,1}-u_{n-1,2} v_{n-2,2}]
\\&\qquad\qquad\qquad    -b k_{n-1} [u_{n-2,1} v_{n-1,1}-u_{n-2,2} v_{n-1,2}].
\end{split}&
\end{align}
It is proved in \cite{HB+BF-12} that the maps $\bb$ for  GOY and
SABRA shell models defined as above satisfy  Assumption \ref{B}.

Let $\mathfrak{P}=(\Omega, \mathcal{F}, \mathbb{P}, \mathbb{F})$
be a filtered complete probability space such that the filtration
 $\mathbb{F}=(\mathcal{F}_t)_{t\ge0}$ satisfies the usual condition.

Let $\eta:=\{\eta_1, \eta_2,\ldots \}$ be a family of mutually
independent Poisson random measures defined on $\mathfrak{P}$ with
L\'evy measures $\{\nu_1, \nu_2,\ldots \}$. We assume that each
$\nu_j$ is a $\sigma$-finite measure on
$\mathbb{R}_0:=\mathbb{R}\backslash\{0\}$. We denote by
$\{\nu_1(dz_1)dt, \nu_2(dz)dt,\ldots \}$ the family of
compensators of the elements of $\eta$ and $\{\tilde{\eta_1},
\tilde{\eta}_2, \ldots\}$ the family of compensated Poisson random
measures associated to the elements of $\eta$. To shorten notation
we will use the following notations
$d\eta_j(z,t):=\eta_j(dz,dt)$, $d \tilde{\eta}_j(z,t):=\tilde{\eta}_j(dz,
dt)$ and $d\nu_j(z)dt:=\nu_j(dz)dt$ for any $j\in \{1,
2,\ldots\}$. We will also use the notation
$$d\bar{\eta}_j(z,t)=\mathds{1}_{\lvert z_j\rvert\le 1} d\tilde{\eta}_j(z,t)+
\mathds{1}_{\lvert z_j\rvert>1} d\eta_j(z,t),$$ for any $j\in \{1,
2,\ldots\}.$
\begin{assum}\label{Lev-Meas}
\begin{enumerate}[(i)]
\item \label{LM0} The Poisson random measures $\eta_j, j \in \{1,
2,\ldots\} $ are independent and identically distributed. This means
in particular that there exists a L\'evy measure $\nu$ such that $
\nu_j(dz_j)=\nu(dz)$ for any $j\in \{1, 2,\ldots \}$.

\item \label{LMi}There exists a $C^1$ function $g: \mathbb{R}
\rightarrow [0,\infty)$ and a constant $C>0$ such that $\nu(dz)=g(z)dz$ and for any $p\ge 1$ we have $$\biggl\lvert
\frac{g^\prime(z)}{g(z)}\biggr\rvert^p\le C(1+\lvert
z\rvert^{-p}), \,\, z\in \mathbb{R}_0.$$

\item \label{LMii} As $\lvert z\rvert \rightarrow \infty$ we have
$z^2 g(z)\rightarrow 0$. Also
$$ \int_{\mathbb{R}_0} \Big[\lvert z\rvert^{q} \mathds{1}_{(\lvert
z\rvert\le 1)} + \lvert z\rvert^{q} \mathds{1}_{(\lvert z\rvert>1
)}\Big]\nu (dz)<\infty,
$$ for any $q\ge 1$.

\item \label{LMiii}Furthermore, there exists a constant $\alpha>0$
such that for any  $y\in \mathbb{R}$
$$ \liminf_{\eps\searrow 0}\eps^\alpha \int_{\mathbb{R}_0} (\lvert
z\cdot y/\eps\rvert^2\wedge 1)\nu(dz)>0.$$
\end{enumerate}
\end{assum}
\begin{Rem}
Assumption \ref{Lev-Meas}\eqref{LMii}-Assumption
\ref{Lev-Meas}\eqref{LMiii} are very similar to \cite[Assumption
1]{Takeuchi} and ensure the validity of a Bismut-Elworthy-Li
formula that we will prove in Appendix \ref{SEC-BEL}. Assumption
\ref{Lev-Meas}\eqref{LMiii} is called in some literature the order and non-degeneracy condition (see \cite[Remark 2.2]{Takeuchi}).
\end{Rem}
The following concept plays an essential role in the proof of our
main result. The following definitions is taken from \cite{Simon}.
\begin{Defs}
\begin{enumerate}[(1)]
\item A real-valued \levy process $\{l(t); t\ge0\}$ has the small
deviation property if for any $T>0$ and $\eps>0$,
\begin{equation*}
\mathbb{P}(\sup_{t\in [0,T]}\lvert l(t)\rvert< \eps)>0.
\end{equation*}
\item A \levy measure $\rho$ on $\mathbb{R}_0$ is said of
\textit{type (I)}  if
\begin{equation*}
\int_{-1}^1 \lvert z\rvert \rho(dz)<\infty.
\end{equation*}

\item A real-valued pure jump \levy process is a \levy process
without continuous part.
\end{enumerate}
\end{Defs}
We recall the characterization of the small deviation property for
real-valued pure jump \levy processes in the following proposition
(see \cite[Th\'eor\`eme, pp 157]{Simon}, \cite[Proposition
1.1]{Dereich}).
\begin{prop}\label{small-dev}
A real-valued pure jump \levy process $\{l(t); t\ge0\}$  admits
the small deviation property if its \levy measure $\rho$ is not of
type (I) or it is of type (I) and, for $\mathcal{E}=-\int_{\lvert
z\rvert\le 1} z\rho(dz)$, we have
\begin{itemize}
\item $\mathcal{E}=0$, or

\item $\mathcal{E}>0$ and $\rho(-\eps\le z<0)\neq 0$ for all
$\eps>0$, or

\item $\mathcal{E}<0$ and $\rho(0< z\le \eps)\neq 0$ for all
$\eps>0$.
\end{itemize}
\end{prop}
\begin{Def}
If a \levy measure $\rho$ on $\mathbb{R}_0$ satisfies the
characterizations given in Proposition \ref{small-dev}, then we
will say that it satisfies the small deviation property condition.
\end{Def}
 Now we introduce an additional assumption for the \levy measure
$\nu$ given in Assumption \ref{Lev-Meas}.
\begin{assum}\label{IID}
The L\'evy measure $\nu$ satisfies the small deviation property
condition (see Proposition \ref{small-dev}).
\end{assum}

Before we state the final assumption for the paper we give some basic examples that satisfy
Assumption \ref{Lev-Meas} and Assumption \ref{IID}.
\begin{Rem}
 \begin{enumerate}
 \item Let $c_+, C_-, \beta_+, \beta_->0$ and $\alpha \in [0,1)$. Define the general tempered \levy measure
  $$ \nu(dz)= c_{+} \lvert z\rvert^{-1-\alpha} e^{-\beta_{+}\lvert z\rvert}\mathds{1}_{z>0}+c_{-} \lvert z\rvert^{-1} e^{-\beta_{-}\lvert z\rvert^{-1}}\mathds{1}_{z<0}.$$
  The simplest choice for the parameters  $c_+, C_-, \beta_+, \beta_->0$ for $\nu$ to satisfy Assumption \ref{Lev-Meas} and Assumption \ref{IID} is
  $$c_+=c_-, \beta_+=\beta_-, \text{ and } \alpha \in [0,1).$$ This choice corresponds to a symmetric tempered stable \levy measure.  This claim can be checked by elementary arguments. With the help of a good software (Mathematica for instance) one can also play with $c_+, C_-, \beta_+, \beta_->0$ and give other choices which are more complicated than the one above. We leave this as an exercise for the interested reader.
 \item The components of the noise in \eqref{Hydro} can be replaced with the following ones
 \begin{equation}
 \ell_k(t)=\sigma W_k(G_k(t))+\theta G_k(t),\quad \sigma>0, \theta \in \mathbb{R}, t\in [0,\infty), k\in \mathbb{N},
 \end{equation}
 where $\{W_k; k\in \mathbb{N}\}$ is a family of i.i.d standard Brownian motions and
 $\{G_k; k\in \mathbb{N}\}$ is a family of i.i.d Gamma processes with \levy measure
 $\nu_G(dz)=(\vartheta z)^{-1} e^{-\frac z\vartheta} \mathds{1}_{z>0} dz$, $\vartheta>0$.
 In fact it was shown in \cite[Chapter 10]{Hilber et al} that each $\ell_k$ is a pure jump \levy noise which is identical in law to a
  variance Gamma process $\tilde{\ell}_k$ having a \levy measure
 $$ \nu(dz)=\left( c \lvert z\rvert^{-1} e^{-\beta_{+}\lvert z\rvert}\mathds{1}_{z>0}+c \lvert z\rvert^{-1} e^{-\beta_{-}\lvert z\rvert}\mathds{1}_{z<0}\right) dz,$$
 with $c=\frac1\vartheta,\,\, \beta_{+}= 2c/(\sqrt{2\sigma^2/\vartheta +\theta^2}+\theta), \beta_{-}= 2c/(\sqrt{2\sigma^2/\vartheta +\theta^2}-\theta).$ If we take $\theta=0$, then we are in the situation of symmetric tempered stable process with $\alpha=0$. We can also play with the parameters $\vartheta$, $\sigma$ and $\theta$ to give other examples, but we again leave it for the interested reader.
  \end{enumerate}
 \end{Rem}

The final assumption on our model is the following.
\begin{assum}\label{Beta}
\item Let $\{\beta_j; j=1, 2, \ldots\}$ be a family of positive
numbers such that there exists $\theta\in (\frac14,\frac12)$ and
$$ \beta_j=\lambda_j^{-\theta}.$$
\end{assum}
To close this section we also introduce the following additional
notations. For a Banach space $\mathrm{B}$ we denote respectively
by $B_b(\mathrm{B})$, $C_b(\mathrm{B})$, and $C^2_b(\mathrm{B})$
the space of bounded and measurable functions, the space of
continuous and bounded functions, and the space of bounded and
twice Fr\'echet differentiable functions on $\mathrm{B}$ and
taking values in $\mathbb{R}$. For two Banach spaces
$\mathrm{B}_1$ and $\mathrm{B}_2$ we denote by
$C^2_b(\mathrm{B}_1, \mathrm{B}_2)$ the space of bounded and
twice Fr\'echet differentiable functions on $\mathrm{B}_1$ and
taking values in $\mathrm{B}_2.$  Throughout this paper, $ B_{\mathrm{B}}(x,r)$ is the ball of radius $r$ centered at $x \in \mathrm{B}$; when $x=0$ we simply write $ B_{\mathrm{B}}(r)$.
\section{Ergodicity of the stochastic Shell models}\label{ERG-SHELL}
 The aim of this
paper is to study the uniqueness of the invariant measure
associated to the solution of the
 abstract evolution equation given by
 \begin{equation}
 \begin{split}\label{Hydro}
  & d\bu(t)+[\kappa \rA\bu(t)+\bb(\bu(t),\bu(t))]dt= \sum_{k=1}^\infty \int_{\mathbb{R}_0} \beta_k z e_k\;d\bar{\eta}_k(z,t)  \\
  & \bu(0)=\xi,
 \end{split}
 \end{equation}
where $\kappa$ is a positive constant and the Poisson random
measures $\eta_j, \;\; j \in \{1,2,\ldots\}$ are as above. In what follows we set $$
L(t)=\sum_{k=1}^\infty \int_0^t \int_{\mathbb{R}_0} \beta_k z e_k\;
d\bar{\eta}_k(z,s)  .$$

We first introduce the notion of solution and give the conditions
under which a solution $\bu$ to Eq. \eqref{Hydro} exists.
\begin{Def}
Let $T>0$ be an arbitrary real number. An $\mathbb{F}$-adapted process $\bu$ is called a solution of Eq. \eqref{Hydro}
 if the following conditions are satisfied:
 \begin{enumerate}[(i)]
  \item $\bu \in L^2(0,T; \h) $ $\mathbb{P}$-almost surely,
  \item the following equality holds for every $t\in [0,T]$ and $\mathbb{P}$-a.s,
  \begin{equation}
   (\bu(t),\phi)=(\xi,\phi)-\int_0^t\left(\langle\kappa \rA\bu(s)+\bb(\bu(s), \bu(s)),\phi\rangle \right)ds+
   \langle  L(t), \phi \rangle,
  \end{equation}
for any $\phi\in \ve$.
\end{enumerate}
\end{Def}
\begin{Rem}
Owing to Assumption \ref{B}-\eqref{B-a} the nonlinear term $ \int_0^t \langle \bb(\bu(s), \bu(s)),\phi\rangle ds$ makes sense whenever $\phi \in \ve$ and $\bu \in L^2(0,T;\h)$. Also, thanks to Assumption \ref{Beta} and \cite[Theorem 4.40]{SP+JZ} it
can be easily checked that the Levy process $L$ lives in
$D(A^{-\frac12})$. Thus $\langle L(t),  \phi \rangle$ makes sense
for any $\phi\in \ve$.
\end{Rem}

\subsection{Resolvability of problem \eqref{Hydro}} We state and
prove the following proposition.
\begin{prop}\label{SHELL}
In addition to Assumption \ref{A}, Assumption \ref{B} and
Assumption \ref{Beta} we assume that the items \eqref{LM0} and
\eqref{LMii} of Assumption \ref{Lev-Meas} hold. Then, problem
\eqref{Hydro} has a unique solution $\bu$ which has a \cadlag
modification in $\h$. For any $t\ge 0$, $\xi\in \h$ and $p\in \{2,4\}$ there
exists a constant $C:=C(t,\xi)>0$ such that
\begin{equation}\label{Eq-9}
\EE \sup_{s\in [0,t]}\lvert \bu(s,\xi)\rvert^{p} +\kappa \EE \int_0^t \lvert \bu(s,\xi)\rvert^{p-2} \lvert
\mathrm{A}^\frac12 \bu(s, \xi)\rvert^2 ds\le C.
\end{equation}
Moreover, $\bu$ is a Markov process having
the Feller property.
\end{prop}
\begin{proof}
First we will prove that \eqref{Eq-9} holds for the Galerkin
approximation of \eqref{Hydro}.

For each $n\in \mathbb{N}$ let $$ \h_n:=\gal\{e_1,\ldots, e_n\},$$
  and $\Pi_n: \ve^\ast \rightarrow \h_n$ be the orthogonal
 projection defined by $$\Pi_n \bv:= \sum_{k=1} \langle \bv, e_k\rangle e_k, \text{ for any } \bv \in \ve^\ast.$$
 Throughout this paper, we will identify $\h_n$ \del{and $\ve_n$ $\ve^\ast_n$} with
 $\mathbb{R}^n$.

Owing to \cite[Theorem 3.1]{Albeverio}, for each $n\in \mathbb{N}$
there exists a \cadlag process $\bun$ which solves the system of
stochastic differential equations
\begin{equation*}
d\bun+[\kappa \mathrm{A}\bun(t)+\Pi_n
\bb(\bun(t),\bun(t))]dt=\sum_{k=1}^n \beta_k dl_k(t)e_k, \,\,
\bun(0)=\Pi_n \xi,
\end{equation*}
where
$$ dl_k(t)=\int_{\lvert z\rvert<1} z
\tilde{\eta}_k(dz,dt)+\int_{\lve z\rve\ge 1} z  \eta_k(dz,dt)=:
\int_{\mathbb{R}_0} zd\bar{\eta}_k(z,t).$$ Applying It\^o's formula to
$\lvert \bun(t)\rve^p$ and using Assumption \ref{B}-\eqref{B-c}
yield
\begin{equation*}
\begin{split}
d \lve \bun(t)\rve^p +p \kappa \lve \mathrm{A}^\frac12
\bun(t)\rve^2 \lve \bun(t)\rve^{p-2}dt = dI_1(t)+
dI_2(t)+I_3(t)dt,
\end{split}
\end{equation*}
where
\begin{align*}
dI_1(t):=\sum_{k=1}^n \int_{\lve z\rve<1} (\lve \bun(t)+\beta_k z
e_k\rve^p-\lve \bun(t)\rve^p)d\tilde{\eta}_k(z,t),\\
dI_2(t):=\sum_{k=1}^n \int_{\lve z\rve\ge 1} (\lve \bun(t)+\beta_k z
e_k\rve^p-\lve \bun(t)\rve^p)d\eta_k(z,t),\\
I_3(t):=\sum_{k=1}^n \int_{\lvert z\rvert<1}(\lve \bun(t)+\beta_k z
e_k\rve^p-\lve \bun(t)\rve^p -\Psi_{ \bun, t}[\beta_kz
e_k] )\nu(dz)dt,
\end{align*}
and $\Psi_{\bun, t}[h]= p \lvert \bun(t)\rvert^{p-2}\langle \bun(t), h\rangle$ for any $h\in \h$ and $t\ge 0$.
Since $\eta_k(dz,dt)$, $k=1,2,\ldots$ are non-negative measures
and
\begin{equation*}
\begin{split}
\lve \bun(t)+\beta_k z e_k\rve^p-\lve \bun(t)\rve^p\le
\beta_k^p \lvert z\rvert^p+C_p \sum_{r=1}^{p-1} \beta_k^{r} \lvert
z\rve^{r} \lve \bun(t)\rve^{p-r}\\
\le \beta_k^p \lvert z\rvert^p+C_p \sum_{r=1}^{p-1} \beta_k^{r} \lvert
z\rve^{r} (1+ \lvert \bun(t)\rvert^p)
\end{split}
\end{equation*}
we have
\begin{equation*}
\begin{split}
\lvert I_2(t)\rvert \le C_p \sum_{k=1}^n \biggl( \int_0^t
\int_{\lve z\rve\ge 1} \left( \sum_{r=1}^{p} \beta_k^{r} \lvert
z\rve^{r}\right) d\eta_k(z,s) + \int_0^t
\int_{\lve z\rve\ge 1}\left( C_p \sum_{r=1}^{p-1} \beta_k^{r} \lvert
z\rve^{r}
\lve \bun(s) \rve^p\right ) d\eta_k(z,s)\biggr).
\end{split}
\end{equation*}
Hence,
\begin{equation*}
\begin{split}
\EE\sup_{s\in [0,t]}\lve I_2(s)\rve&\le C_p \sum_{k=1}^n \EE \biggl( \int_0^t
\int_{\lve z\rve\ge 1} \left( \sum_{r=1}^{p} \beta_k^{r} \lvert
z\rve^{r}\right) d\eta_k(z,s)\\ & \quad + \int_0^t
\int_{\lve z\rve\ge 1}\left( C_p \sum_{r=1}^{p-1} \beta_k^{r} \lvert
z\rve^{r}
\lve \bun(s) \rve^p\right ) d\eta_k(z,s)\biggr) \\
& \le t C_p C_\nu \sum_{r=1}^ p\Big[\sum_{k=1}^\infty \beta_k^r\Big]+ C_p C_\nu \sum_{r=1}^ p\Big[\sum_{k=1}^\infty \beta_k^r\Big] \times
\EE \int_0^t \lvert \bun(s)\rvert^{p}ds
,
\end{split}
\end{equation*}
where for any $q\ge 1$ we have set $C_\nu= \max_{q\ge 1}\Big[\int_{\mathbb{R}_0}\mathds{1}_{(\lvert z\rvert<1)} \lvert z\rvert^{q} \nu(dz) +
\int_{\mathbb{R}_0}\mathds{1}_{(\lvert z\rvert\ge1)} \lvert z\rvert^q \nu(dz)\Big].$
Since  $C_\beta:= \sum_{k=1}^\infty \beta_k^\alpha <\infty$, $C_\nu<\infty$ for any
$\alpha>0$ and $q\ge 1$, respectively, we derive that there exists a constant $C>0$ such that
\begin{equation*}
\EE\sup_{s\in [0,t]}\lve I_2(s)\rve\le C\,t + C\,  \EE\int_0^t \lve \bun(s)\rve^p ds.
\end{equation*}
Using Burkholder-Davis-Gundy inequality and similar idea  as above, we deduce that
\begin{align}
\EE\sup_{s\in [0,t]}\lve I_1(s)\rve&\le c_p \sum_{k=1}^n \EE
\biggl[\int_0^t \int_{\lve z\rve<1} \left(\lve \bun(t)+\beta_k z
e_k\rve^p-\lve \bun(t)\rve^p\right)^2\nu(dz)dt\biggr]^\frac12\nonumber \\
&\le C_p \sum_{k=1}^n \EE \biggl[\int_0^t \int_{\lve z\rve<1}
\left(\beta_k^p \lvert z\rvert^p+ \beta_k \lvert
z\rve \lvert \bun(t)\rvert^{p-1}+\sum_{r=1}^{p-2} \beta_k^{r} \lvert
z\rve^{r} \lvert \bun(t)\rvert^{p-r}\right)^2\nu(dz)dt\biggr]^\frac12\nonumber\\
& \le  C_p \sum_{k=1}^n \EE \biggl[\int_0^t \int_{\lve z\rve<1}
\left(\sum_{r=1}^p \beta_k^{2r} \lvert z\rvert^{2r}+ \sum_{r=1}^{p-1}\beta_k^{2r} \lvert
z\rve^{2r} \lve \bun(s)\rve^{2(p-1)}\right)\nu(dz)ds\biggr]^\frac12\nonumber\\
&\le C_p C_\beta C_\nu t\left(1 +\EE\sup_{s\in [0,t]}\lve \bun(s)\rve^{p-1}\right). \label{YOUNG}
\end{align}
where we understand that $\sum_{r=2}^{p-1} \beta_k^{r} \lvert
z\rve^{r} \lvert \bun(t)\rvert^{p-r}=0$ if $p=2.$
Recall that for any real numbers $a\ge 0$ and $b\ge 0$ we have  $$ ab\le \frac{a^p}{p \eps^p}+ \frac{p-1}{p} (b\eps)^{\frac{p}{p-1}},$$ for any $\eps>0$. We deduce from this and the inequality \eqref{YOUNG} that there exists $C>0$ such that
$$ \EE\sup_{s\in [0,t]}\lve I_1(s)\rve\le  C\, t +\frac12 \EE\sup_{s\in [0,t]}\lve \bun(s)\rve^p. $$
Now we deal with $I_3(t)$. As above, it is easy to see that 
\begin{equation*}
\begin{split}
\lve \lve \bun(t)+\beta_k z e_k\rve^p-\lve \bun(t)\rve^p -\Psi_{\bun, t}[\beta_kz e_k] \rve\le C_p\left(\beta_k^p \lvert z\rvert^p+ \sum_{r=1}^{p-1} \beta_k^{r} \lvert
z\rve^{r} \lve \bun(t)\rve^{p-r}+
p\lve \bun\rve^{p-1}\beta_k \lve z\rve \right).
\end{split}
\end{equation*}
Thus, using Young's inequality we easily deduce
that
\begin{align*}
\lve \lve \bun(t)+\beta_k z e_k\rve^p-\lve \bun(t)\rve^p -\Psi_{\bun, t}[\beta_kz e_k] \rve\le C_p\left(\sum_{r=1}^p \beta_k \lve z\rve^r + \sum_{r=1}^{p-1} \beta_k^r \lve z\rve^r \lve \bun(t) \rve^p\right).
\end{align*}
Hence, arguing as in the case of $I_2$ we deduce that there exists a constant $C>0$ such that
\begin{equation*}
\EE\sup_{s\in [0,t]}\biggl\lve \int_0^s I_3(r)dr\biggr\rve\le C\, t + C\, \EE\int_0^t \lve \bun(s)\rve^p
ds.
\end{equation*}
Summing up we have showed that there exists $C>0$ such that for
any $n\in \mathbb{N}$
\begin{equation*}
\frac12 \EE\sup_{s\in [0,t]} \lve \bun(s)\rve^p +p\kappa \EE \int_0^t \lve
\mathrm{A}^\frac12 \bun(s)\rve^2 \lve \bun(s)\rve^{p-2}ds\le \lve
\xi\rve^p+ C t+ C \EE\int_0^t \lve \bun(s)\rve^p ds.
\end{equation*}
Invoking the Gronwall's inequality we infer that there exists
$K_0$ and $K_1$  such that for any $n\in \mathbb{N}$
\begin{equation}\label{Eq-10}
\EE\sup_{s\in [0,t]}\lve \bun(s)\rve^p+2p\kappa \EE\int_0^t \lve
\bun(s)\rve^{p-2} \lve \mathrm{A}^\frac 12 \bun(s)\rve^2 ds\le
(K_0 t+\lve \xi\rve^p)e^{K_1t}.
\end{equation}
Now, the existence of solution $\bu$ will follow from a similar
argument as in \cite{EH+PR} (see also \cite{EH+PAR+MS},
\cite{Motyl}). The uniqueness of the solution can be proved by
arguing as in \cite{EH+PR} or \cite{ZB+EH+JZ}.

By Assumption \ref{A}, $ \langle \rA y,
y\rangle=\vert\rA^\frac12 y\rvert^2$ for any $y\in \ve$, thus thanks to Assumption \ref{B} we can argue as in \cite{EH+PR} and show that
for any $t\ge 0$ we have
\begin{align}
\lim_{n\to \infty} \mathbb{E}\lvert \bun(t)-\bu(t)\rvert^2=0,\label{t0-1}\\
\lim_{n \to \infty}\mathbb{E} \int_0^t \lvert \rA^\frac12[
\bun(s)-\bu(s)]\rvert^2 ds=0.\label{t0-2}
\end{align}
Now, we prove that $\bu$ has a \cadlag modification in $\h$. Our
proof relies very much on recent result about \cadlag property of
stochastic convolution proved in \cite{SP+JZ-13}. Let
$\mathfrak{S}$ be the stochastic convolution defined by
\begin{equation}\label{STOC-CONV-1}
\mathfrak{S}(t)=\sum_{k=1}^\infty \mathfrak{S}_k(t) e_k,
\end{equation}
where each $\mathfrak{S}_k$ is the solution to
\begin{equation}\label{STOC-CONV-2}
d\mathfrak{S}_k(t)=-{\kappa} \lambda_k \mathfrak{S}_k(t) dt+\beta_k
\int_{\mathbb{R}_0} z d\bar{\eta}_k(z,t).
\end{equation}
Since, by Assumption \ref{A}-\eqref{Aii} and Assumption
\ref{Beta},
$$ \sum_{k=1}^\infty\left( \lambda_k^{\eps-1} \beta_k^2
+\lambda_k^\eps \beta^4_k\right)\le C \sum_{k=1}^\infty
\left(\lambda^{-2k(\eps-[1+2\theta])}+\lambda^{(\eps-4\theta)2 k}
\right)<\infty,$$ for any $\eps \in (0, 2)$, it follows from
\cite[Corollary 3.3]{SP+JZ-13} that $\mathfrak{S}$
has a  \cadlag modification in $\h$. Let us also consider the
following problem
\begin{equation}\label{stoc+v}
\frac{d\bv(t)}{d t} +\kappa \rA \bv(t)
+\bb(\bv(t)+\mathfrak{S}(t), \bv(t)+\mathfrak{S}(t))=0,\,\,
\bv(0)=\xi\in \h,
\end{equation}
where $\mathfrak{S} \in L^\infty(0,T;\h)$. Arguing as in Appendix
\ref{APP-CONV} we can show that it has a unique solution $\bv \in
C(0,T;,\h)\cap L^2(0,T;\ve)$. Taking $\mathfrak{S}$ as the
stochastic convolution defined in
\eqref{STOC-CONV-1}-\eqref{STOC-CONV-2} we see, thanks to the
uniqueness of solution, that $\bu=\bv+\mathfrak{S}$ solves
\eqref{Hydro}. Thanks to \cite[Theorem 4.1]{Whitt} the function
$\phi: C([0,T];\h)\times \mathbb{D}([0,T];\h) \ni (x,y)  \mapsto
x+y \in \mathbb{D}([0,T];\h)$ is continuous, hence $\bu$ has a
\cadlag modification in $\h$. Here $\mathbb{D}([0,T];\h)$ denotes
the space, equipped with the Skorokhod topology $J_1$,  of \cadlag
functions taking values in $\h$.

 Arguing as in  \cite[Section 6]{AMR-09} we can show
that $\bu$ is a Markov semigroup. The idea in \cite{EH+PR} can be
used to prove that $\bu$ has the Feller property.
\end{proof}

\subsection{Uniqueness of the invariant measure for the stochastic
Shell models} The preparatory result in the previous subsection
enables us to define a Markov semigroup  which is generated by the
Markov solution $\bu$ to \eqref{Hydro}. More precisely, we can
define a Markov semigroup as in the following definition.
 \begin{Def}\label{MAP}
 Let $\{\CP_t; t\ge 0\}$ be the Markov semigroup
defined by
$$[\CP_t\Phi](\xi)=\EE[\Phi(\bu(t,\xi))],\,\, \Phi\in B_b(\h), \xi
\in \h,\,\, t\ge 0,$$ where $\bu(\cdot, \xi)$ is the unique solution to
\eqref{Hydro} with initial condition $\xi \in \h$. For simplicity
we will write $$ \CP_t\Phi(\xi):=[\CP_t\Phi](\xi),\,\, \Phi\in
B_b(\h), \xi \in \h,\,\, t\ge 0,$$ throughout.
\end{Def}
We will establish that the Markov semigroup $\{\CP_t; t\ge 0\}$
has a unique invariant measure which then implies the ergodicity of
the solution to \eqref{Hydro}.

First let us introduce an auxiliary problem. For this aim, let $R\in (0, \infty)$
and $\rho(\cdot):[0,\infty)\rightarrow
[0,1]$ be a $C^\infty$ and Lipschitz function such that
\begin{equation*}
\rho(x)=
\begin{cases}
1 \text{ if } x\in [0,1],\\
0 \text{ if } x\in [2,\infty],\\
\in [0,1] \text{ if } x\in [1,2],
\end{cases}
\end{equation*}
and $ \lvert \rho^\prime(x)\lvert\le 2$. For any $\bu \in \h$ let $\{\bb^R (\bu,\bu): n \in \mathbb{N} \}$
be the family defined by
$$\bb^R(\bu, \bu):=\rho\left(\frac{\lvert \bu\rvert^2
}{R}\right)\bb(\bu,\bu), \text{ for any } R \in \mathbb{N}.$$ Let
us consider the following modified problem
\begin{equation}\label{Hydro-1}
 \begin{split}
  & d\bu^R(t)+[\kappa \rA\bu^R(t)+\bb^R(\bu^R(t),\bu^R(t))]dt=\sum_{k=1}^\infty \beta_k dl_k(t)e_k,\\
  & \bu^R(0)=\xi \in \h.
 \end{split}
 \end{equation}
 We have the following results which will be proved in the
 next section.
 \begin{prop}\label{TSF-MSM}
Let Assumption \ref{A}, Assumption \ref{B}, Assumption \ref{Beta}
and Assumption \ref{Lev-Meas} hold. Then for each $R>0$ and
$\xi\in \h$ the problem \eqref{Hydro-1} has a unique solution
$\bu^R:=\bu^R(\cdot,\xi)$. The stochastic process $\bu^R$ is a
Markov process which has the Feller property. Furthermore, if
$\{\CP^R_t; t\ge 0\}$ denotes the Markov semigroup associated to $\bu^R$ (see
Definition \ref{MAP}) then for any $R>0$, $t> 0$ there exists a
positive constant $C:=C(t,R)$ such that
\begin{equation}\label{SF-MSM-R}
\lvert \CP^R_t \Phi(\xi)-\CP^R_t  \Phi(\zeta)\rvert< C \Lve \Phi
\Rve_\infty \lvert \xi-\zeta\rvert,
\end{equation}
for any $\xi, \zeta\in \h$ and $\Phi\in B_b(\h)$.
 \end{prop}
Now, let us state two
 propositions whose proofs will be given below. The following
 proposition shows that the Markov semigroup associated to the
 unique solution of \eqref{Hydro} has a certain smoothing
 property.
\begin{prop}\label{TRU-CUT}
Let Assumption \ref{A}, Assumption \ref{B}, Assumption \ref{Beta}
and Assumption \ref{Lev-Meas} hold. Let $\{\CP_t; t\ge0\}$, be the
Markov process associated to the solution $\bu$ of \eqref{Hydro}.
Then it has the strong Feller
property, i.e., $\CP_t \mathcal{B}_b(\h)\subset \mathcal{C}_b(\h)$ for any $t> 0$.
\end{prop}
According to \cite[Theorem 0.3]{MH-08}, for the invariant measure
to be unique it is sufficient to find a point $\xi\in \h$ that is
accessible for $\CP_t$. The definition of an accessible point is
given in the following definition (see, for instance,
\cite{DP+JZ-96}, \cite{MH-08}).
\begin{Def}
Let $\mathcal{R}_\lambda$ be the resolvent of $\CP_t$ defined by
\begin{equation*}
\mathcal{R}_\lambda(\xi,\mathcal{U})=\int_0^\infty
e^{-\lambda t}[\CP_t\mathds{1}_\mathcal{U}](\xi) dt,
\end{equation*}
for any measurable set $\mathcal{U}\subset \h$, $\lambda >0$ and
$\xi\in \h$. A point $\mathbf{x} \in \h$ is accessible if, for
every $\xi\in \h$ and every open neighborhood $\mathcal{U}$ of
$\mathbf{x}$, one has $\mathcal{R}_\lambda(\xi, \mathcal{U})>0$.
\end{Def}
For our model we have the following result.
\begin{prop}\label{IRRED}
In addition to the assumptions of Proposition \ref{TRU-CUT}
suppose also that Assumption \ref{IID} holds.  Then, the point
$0\in \h$ is accessible for $\{\CP_t; t\ge 0\}$.
\end{prop}
Before we proceed to the statement and the proof of the main result of this paper we should give a refinement of the estimate \eqref{Eq-9}
in Proposition \ref{SHELL}.
\begin{lem}
There exists a constant $C>0$ such that for any $T>0$ and $\xi \in \h$ we have
\begin{align}
\EE\lvert \bu(T, \xi)\rvert^2 \le (\lvert \xi\rvert^2+C T) e^{-\frac\kappa\lambda_1T},\label{New-EST-1}\\
\EE \int_0^T \lvert A^\frac12 \bu(s,\xi)\rvert^2 ds\le (\lvert \xi\rvert^2+C T) +C\frac{\lambda_1}{\kappa}\label{New-EST-2}
\end{align}
\end{lem}
\begin{proof}
We argue as in the proof of Proposition \ref{SHELL} and work with the Galerkin approximation. Note that thanks to the estimate \eqref{Eq-9} the stochastic
process $$ M_n(t)=\sum_{k=1}^n \int_0^t\int_{\lvert z\rvert\le 1} [\beta_k^2 \lvert z\rvert^2+2\beta_k \langle \bun(s), e_k\rangle z]d\tilde{\eta}_k(z,s), $$ is a martingale satisfying $\EE M_n(t)=0$ for any $t>0$.   Therefore, arguing as before we derive the following chain of equalities/inequalities
\begin{equation*}
\begin{split}
\EE \lvert \bun(t)\rvert^2 +2 \kappa \EE\int_0^t \lvert A^\frac12 \bun(s)\rvert^2 ds=\lvert \xi \rvert^2 + \sum_{k=1}^n \EE \int_0^t \int_{\lvert z\rvert^2 >1}[\beta_k^2 \lve z\rve^2 + 2\beta_k \langle \bun(s), e_k\rangle z] d\eta_k(dz,ds)\\ + 2t  \left(\int_{\mathbb{R}_0}\lve z\rve^2 \nu(dz)\right)\sum_{k=1}^n \beta_k^2 \\
\le \lve \xi\rve^2 + 4 C_\beta C_\nu t + 2\sum_{k=1}^n \int_0^t \int_{\lve z\rve>1}\EE \lve \bun(s)\rve\lve z\rve \nu(dz)ds\\
\le \lve \xi\rve^2 + 4 C_\beta C_\nu t + 2C_\beta C_\nu \int_0^t \EE \lve \bun(s)\rve ds.
\end{split}
\end{equation*}
Thanks to the Poincar\'e inequality \eqref{Poincare} and the fact that  $\lve \bun(s)\rve\le \frac{\lambda_1^2}{4 \kappa}+\frac{\kappa}{\lambda_1^2 }\lve \bun(s)\rve^2$, we derive from the chain of inequalities above that
\begin{equation*}
\EE \lvert \bun(t)\rvert^2\le  \lve \xi\rve^2 + 2 C_\beta C_\nu(2+\frac{\lambda_1^2}{\kappa}) t -\frac{\kappa}{\lambda_1^2}\int_0^t \EE \lve \bun(s)\rve^2 ds,
\end{equation*}
which implies that
\begin{equation*}
\EE \lve \bun(t)\rve^2 \le (\lve \xi\rve^2+ 2 C_\beta C_\nu(2+\frac{\lambda_1^2}{\kappa}) t)e^{-\frac{\kappa}{\lambda_1^2}t},\quad  t>0, \,\, n\in \mathbb{N}.
\end{equation*}
Observe that 
\begin{equation*}
\begin{split}
\EE \lvert \bun(t)\rvert^2 +2 \kappa \EE\int_0^t \lvert A^\frac12 \bun(s)\rvert^2 ds
\le \lve \xi\rve^2 + 4 C_\beta C_\nu t + 2C_\beta C_\nu \int_0^t \EE \lve \bun(s)\rve ds.
\end{split}
\end{equation*}
Therefore, there exists $C>0$ such that
\begin{equation*}
\EE \int_0^t \lve A^\frac12 \bun(s)\rvert^2 ds \le C(\lve \xi\rve^2+1)t +C \int_0^t se^{-\frac{\kappa}{\lambda_1^2}s} ds, \quad \forall t>0, \,\, n\in \mathbb{N}.
\end{equation*}
From the last two estimates, \eqref{t0-1} and \eqref{t0-2} we easily derive the proof of the lemma.
\end{proof}
 Now, we give in the next theorem the main result of the present work.
\begin{thm}\label{Main result}
 Let the assumptions of Proposition \ref{IRRED} holds. Then, the semigroup $\{\CP_t; t\ge 0\}$ admits a unique
invariant measure $\mu$ whose support is included in $\ve$.
\end{thm}
\begin{proof}
Owing to the compact embedding $\ve\subset \h$ and the estimates
\eqref{New-EST-1} and \eqref{New-EST-2} the existence of an invariant measure $\mu$ follows
from the Krylov–Bogolyubov theorem (see, for instance, the proofs
in \cite[Theorem 2.2]{Chow+Kashminskii} or \cite[Theorem
5.3]{EH+PR}). One can also argue as in \cite[Theorem 5.3]{EH+PR}
to show that the support of $\mu$ is included in $\ve$.

Thanks to Proposition \ref{TRU-CUT} and Proposition \ref{IRRED} we
 infer from \cite[Theorem 0.3]{MH-08} that the invariant
 measure $\mu$ is unique.
\end{proof}
Before we proceed to the proofs of our results we state the following remark.
\begin{Rem}
 It is clear from  Assumption \ref{Beta}  that the noise we consider in this paper is not cylindrical and it is not known whether our results hold for the stochastic shell models with cylindrical pure jump \levy noise.
\end{Rem}
Now, we give the proofs of Proposition \ref{TRU-CUT} and
\ref{IRRED}.
\begin{proof}[Proof of Proposition \ref{TRU-CUT}]
We will show that for any $\xi \in \h$, $\eps>0$ there exists a constant $\delta>0$ such that
\begin{equation*}
 \begin{split}
  \lvert \CP_{t}\Phi(\xi)-\CP_{t} \Phi(\zeta)\rvert
  \le  \eps,
 \end{split}
\end{equation*}
for any $\Phi \in B_b(\h)$ with $\Lve \Phi\Rve_\infty\le 1$, $\zeta \in B_\h(\xi,\delta)$.

For this purpose, let $\Phi \in B_b(\h)$ and $R>R_0$ where $R_0$ will be fixed later. Let $ \bu(\cdot, \xi)$ and $\bu(\cdot, \zeta)$ be solutions of \eqref{Hydro} with the initial conditions $\xi\in \h$ and $\zeta \in \h$, respectively. For any $\xi\in \h$ let $\{\tau_R(\xi); R >0\}$ be the family of stopping times
 defined by
 $$ \tau_R(\xi) :=\inf\{t\ge 0; \lvert \bu(t, \xi)\rvert\ge R \}.$$
  Since, by definition of $\bu^R$ and uniqueness of solution of \eqref{Hydro} (see also Remark \ref{REM-SECT-4}), $\bu(t, \cdot)=\bu^R(t,\cdot)$ on $\{t\le \tau_R(\cdot) \}$,  we obtain that
  \begin{equation*}
 \begin{split}
  \lvert \CP_t\Phi(\xi)-\CP_t \Phi(\zeta)\rvert \le & \lvert \CP_t\Phi(\xi)-\CP^R_t
  \Phi(\xi)\rvert+\lvert \CP^R_t\Phi(\xi)-\CP^R_t \Phi(\zeta)\rvert+\lvert \CP^R_t\Phi(\zeta)-\CP_t
  \Phi(\zeta)\rvert\\
  \le & 2 \Lve \Phi\Rve_\infty \mathbb{P}(\tau_R(\xi)<t )+ 2 \Lve \Phi\Rve_\infty \mathbb{P}(\tau_R(\zeta)
  <t)+\lvert \CP^R_t\Phi(\xi)-\CP^R_t \Phi(\zeta)\rvert.
 \end{split}
\end{equation*}
For any $t> 0$, $p\in \{2,4\}$ and $R>0$ we have 
\begin{align*}
\mathbb{E}\lvert \bu(t\wedge \tau_R,\xi)\rvert^p =\mathbb{E}\left(\vert \bu(\tau_R,\xi) \vert^p \mathds{1}_{\tau_R<t}\right) + \mathbb{E}\left(\vert \bu(t,\xi) \vert^p \mathds{1}_{t\le \tau_R}\right).
\end{align*}
From the \cadlag property of $\bu(\cdot,\xi)$ and the definition of $\tau_R$ it follows that $\mathbb{E} \vert \bu(\tau_R,\xi)\vert^p\ge R^p $. Thus,
$$ \mathbb{E}\lvert \bu(t\wedge \tau_R,\xi)\rvert^p\ge R^p \mathbb{P}(\tau_R<t).$$ 
By Proposition \ref{SHELL}, for any $t>0$ and $\xi \in \h$
there exists $C(t,\xi)$ such that  $$\mathbb{E}\sup_{s\in [0,t]}\lvert
\bu(s,\xi)\rvert^p <C(t,\xi),$$ from which and the former estimate we can deduce that for any
$\xi\in \h$ and $t>0$
\begin{equation}
\mathbb{P}(\tau_R(\xi)<t )\le \frac{1}{R^p} C(t,\xi).
\end{equation}
Thanks to Proposition \ref{TSF-MSM} the Markov semigroup $\{ \CP^R_t, t\ge 0\}$ has the strong
Feller property. In particular, we infer from \eqref{SF-MSM-R}
that for any $\Phi\in B_b(\h)$, $\xi \in \h$ we have
\begin{equation*}
 \begin{split}
  \lvert \CP_{t}\Phi(\xi)-\CP_{t} \Phi(\zeta)\rvert
  \le \Lve \Phi\Rve_{\infty} \biggl[\frac{4 C(t,\xi)}{R^p}+C(R,t) \lve \xi-\zeta\rvert\biggr].
 \end{split}
\end{equation*}
Choosing $$R\ge \biggl(\frac{8 C(t,\xi)}{\eps}\biggr)^\frac1p
\text{ and } \delta\le \frac{\eps}{2 C(R,t)},$$ we derive that
 \begin{equation*}
 \begin{split}
  \lvert \CP_{t}\Phi(\xi)-\CP_{t} \Phi(\zeta)\rvert
  \le  \eps,
 \end{split}
\end{equation*}
for any
$\xi\in \h$, $\zeta \in B_\h(\xi,\delta)$, and $\Phi \in B_b(\h)$ with $\Lve \Phi\Rve_\infty\le 1$. This proves that the Markov semigroup $\{\CP_t; \; t\ge 0\}$ is
strong Feller.
\end{proof}

\begin{proof}[Proof of Proposition \ref{IRRED}]
The first step of the proof is to check the following claim. Fixed
any $\eps>0$ and for any $t>0$ define
$$\Omega^\ast=\{\omega\in \Omega: \sup_{s\in [0,t]} \lvert
\mathfrak{S}(s,\omega)\rvert^2< \eps\}$$
where $\mathfrak{S}$ is the stochastic convolution defined by \eqref{STOC-CONV-1}-\eqref{STOC-CONV-2}.

\noindent \textbf{Claim I}: \textit{For any $\eps>0$ and $t>0$ we
have $\mathbb{P}(\Omega^\ast)>0$.}

To prove this claim we first observe that
$$ \mathbb{P}(\Omega^\ast)\ge \mathbb{P} (\Omega^\ast_{1,N}) \cdot
\mathbb{P} (\Omega^\ast_{2,N}),$$ where $N$ is a certain large
integer and $$ \Omega^\ast_{1,N}:=\{\omega: \sup_{s\in
[0,t]}\sum_{k\ge N} \lvert \mathfrak{S}_k(s,\omega)\rvert^2 <\frac\eps
2\},$$
$$ \Omega^\ast_{2,N}:=\{\omega: \sup_{s\in
[0,t]}\sum_{k< N} \lvert \mathfrak{S}_k(s,\omega)\rvert^2 <\frac\eps
2\}.$$ Let
$$ l_k(t)=\int_0^t \int_{\mathbb{R}}z\bar{\eta}_k(z,s),\,\, t\ge 0,
k\in \mathbb{N}.$$
In the remaining part of the  proof we use without further notice the shorthand notations $\mathfrak{S}(s) $ and $l_k(s)$ to denote $\mathfrak{S}(s,\omega)$ and $l_k(s,\omega)$, respectively.

 For each $k$ the process $l_k$ defines a \levy
process with \levy measure $\nu$. For each $t\ge 0$ and $k\in
\mathbb{N}$ the function $e^{-\lambda_k(t-\cdot)}$ is
differentiable, we can apply \cite[Proposition 9.16]{SP+JZ} to
derive that
\begin{equation*}
\mathfrak{S}_k(t)=\left(l_k(t)-{ \kappa}\lambda_k \int_0^t
e^{-\lambda_k(t-s)}l_k(s)ds\right){ \beta_{k}},
\end{equation*}
for any $t>0$ and $k\in \mathbb{N}$. From this last identity we
easily infer that
\begin{equation*}
\sup_{s\in [0,t]}\lvert \mathfrak{S}_k(s)\rvert\le (1+\kappa)\beta_k \sup_{s\in
[0,t]}\lvert l_k(s)\rvert,
\end{equation*}
for any $t>0$ and $k\in \mathbb{N}$. Thanks to the inequality
\begin{equation*}
\sup_{s\in [0,t]}\sum_{k\le N} \lvert \mathfrak{S}_k(s)\rvert^2
\le \sum_{k\le N} \sup_{s\in [0,t]}\lvert
\mathfrak{S}_k(s)\rvert^2\le (1+\kappa)^2 \sum_{k\le N}\beta_k^2 \sup_{s\in
[0,t]}\lvert l_k(s)\rvert^2,
\end{equation*}
we have that $$ \mathbb{P}(\Omega^\ast_{2,N})\ge
\mathbb{P}\biggl(\sum_{k\le N}\beta_k^2 \sup_{s\in [0,t]}\lvert
l_k(s)\rvert^2<\frac\eps{2(1+\kappa)^2}\biggr).$$ Since
$$\{\omega; \sup_{s\in [0,t]} \lvert l_k(s)\rvert <
\frac{1+\kappa}{\beta_k} \sqrt{\frac{\eps}{2N} }\text{ for all } k\le
N\}\subset \{\omega; \sum_{k\le N}\beta_k^2\sup_{s\in [0,t]}\lvert
l_k(s)\rvert^2<\frac\eps {2(1+\kappa)^2}\}$$ and the elements of $\{l_k; k\le
N\}$ are mutually independent, we obtain that
$$ \mathbb{P}(\Omega^\ast_{2,N})\ge
\mathbb{P}\biggl(\sum_{k\le N}\beta_k^2 \sup_{s\in [0,t]}\lvert
l_k(s)\rvert^2<\frac\eps {2(1+\kappa)^2}\biggr)\ge \prod_{k\le N}
\mathbb{P}\left(\sup_{s\in [0,t]} \lvert l_k(s)\rvert <
\frac{1+\kappa}{\beta_k} \sqrt{\frac{\eps}{2N} }\right).$$ Since  the \levy
measure $\nu$ satisfies Assumption \ref{IID} we easily derive that
$\mathbb{P}(\Omega^\ast_{2,N})>0$.
 Since the stochastic convolution $\mathfrak{S}$ is \cadlag
and taking values in $\h$  it follows from \cite[Theorem
2.3]{SP+JZ-13} that there exists an integer $N_0>0$ such that
$\mathbb{P}(\Omega^\ast_{1,N})>0$ for any $N\ge N_0$.
 Now we easily conclude that $\mathbb{P} (\Omega^\ast_{1,N}) \cdot
\mathbb{P} (\Omega^\ast_{2,N})>0$ and thus the proof of the
\textbf{Claim I}.

Now we pass to the next step of the proof of Proposition
\ref{IRRED}. Before proceeding further we introduce a notation.
For any fixed $\delta >0$ and $T>0$ set
$$ \tilde{\Omega}^\ast(\delta, T)= \{\omega\in \Omega: \sup_{s\in [0,T] \lvert
\mathfrak{S}(s,\omega)\rvert^2}< \min\left(\delta,
\frac{\kappa^2}{4 \lambda_1C_0^2}\right)\},$$  where $C_0$ is the
positive constant from Assumption \ref{B}-\eqref{B-a}. The next
step of the proof is to check the validity  of the following
claim.

\noindent \textbf{Claim II}:\textit{ For any $R>0$ and $\gamma>0$
there exist $T_0>0$ and $\delta_0>0$ such that for any $t\ge T_0$,
$\xi \in B_\h(R)$  and $\omega\in \tilde{\Omega}^\ast(\delta_0,
T_0)$
\begin{equation*}
\lvert \bu(t,\omega)\rvert^2\le \gamma.
\end{equation*}
}

\noindent To check this claim we closely follow \cite{EM01}. We
multiply \eqref{stoc+v} by $\bv$ in the scalar product of $\h$ and
obtain
\begin{equation*}
\frac12 \frac{d}{dt}\lvert \bv(t)\rvert^2 +\kappa \lvert
A^{\frac12} \bv(t)^2 =\langle \bb(\bv(t),
\mathfrak{S}(t))+\bb(\mathfrak{S}(t), \mathfrak{S}(t)),
\bv(t)\rangle,
\end{equation*}
where we used Assumption \ref{B}-\eqref{B-c}. Using Assumption
\ref{B}-\eqref{B-a} and Cauchy's inequality we derive that
\begin{align*}
 \frac{d}{dt}\lvert \bv(t)\rvert^2 +2 \kappa \lvert A^{\frac12}
\bv(t)^2 \le 2 \kappa^{-1} C_0^2[ \lvert \bv(t)\rvert^2 \lvert
\mathfrak{S}(t)\rvert^2+\lvert \mathfrak{S}(t)\rvert^4]+\kappa
\lvert A^{\frac12} \bv(t)\rvert^2\\
\frac{d}{dt}\lvert\bv(t)\rvert^2 \le 2 \kappa^{-1} C_0^2[ \lvert
\bv(t)\rvert^2 \lvert \mathfrak{S}(t)\rvert^2+\lvert
\mathfrak{S}(t)\rvert^4]-\kappa \lvert A^{\frac12} \bv(t)\rvert^2.
\end{align*}
Using the inequality \eqref{Poincare} we obtain
\begin{equation*}
\frac{d}{dt}\lvert\bv(t)\rvert^2 \le [2 \kappa^{-1} C_0^2  \lvert
\mathfrak{S}(t)\rvert^2- \lambda^{-1}\kappa ]\lvert
\bv(t)\rvert^2+ 2 \kappa^{-1} C_0^2   \lvert
\mathfrak{S}(t)\rvert^4
\end{equation*}
Using the Gronwall's inequality we infer that on
$\tilde{\Omega}^\ast(\delta, T)$ we have
\begin{equation*}
\lvert \bv(t)\rvert^2\le \lvert \xi\rvert^2
e^{-\frac{\kappa}{2\lambda_1}t}+2\kappa^{-1}C_0^2
\left(\min\left(\delta, \frac{\kappa^2}{4
\lambda_1C_0^2}\right)\right)^2.
\end{equation*}
Thus, for any $R>0$ and $\gamma>0$ we can find $T_0$ and
$\delta_0$ such that on $\tilde{\Omega}^\ast(\delta_0, T_0)$
\begin{equation*}
\lvert \bv(t)\rvert^2 \le \frac\gamma 4,
\end{equation*}
for any $t \ge T_0$ and $\xi \in B_\h(R)$. Choosing $\delta_0$
small we can assume that on $\tilde{\Omega}^\ast(\delta_0, T_0)$
$$ \lvert \mathfrak{S}(t)\rvert^2 \le \frac\gamma 4, \text{ for any } t\ge
T_0.$$ Thus, for any $R>0$ and $\gamma>0$ we found $T_0>0$ and
$\delta_0$ such that on $\tilde{\Omega}^\ast(\delta_0, T_0)$
\begin{equation*}
\lvert \bu(t)\rvert^2=\lvert \bv(t)+\mathfrak{S}(t)\rvert^2
\le \gamma,
\end{equation*}
for any $t\ge T_0$ and $\xi \in B_\h(R)$. This completes the proof
of \textbf{Claim II}.

Now we finalize the proof of Proposition \ref{IRRED}.  We easily
infer from \textbf{Claim II} that for any $R>0$ and $\gamma>0$
there exist two positive constants $T_0$, $\delta_0$ such that for any $t\ge T_0$ and
$\xi\in B_\h(R)$
$$ [\CP_t \mathds{1}_{B_\h(\gamma)}](\xi)\ge \mathbb{P}(\tilde{\Omega}^\ast(\delta_0,
T_0)).$$ Since, by \textbf{Claim I}, we know that
$\mathbb{P}(\tilde{\Omega}^\ast(\delta_0, T_0))>0$ and $R>0$ is arbitrary, we deduce that
for any $t\ge T_0$ and $\xi\in \h$
\begin{equation*}
 [\CP_t \mathds{1}_{B_\h(\gamma)}](\xi)>0.
\end{equation*}
This implies that for every $\xi\in \h$ and every open
neighborhood $\mathcal{U}$ of $0$, one has
$\mathcal{R}_\lambda(\xi, \mathcal{U})>0$,
from which we conclude the proof of Proposition \ref{IRRED}.
\end{proof}
\section{Analytic study of the modified stochastic shell model
\eqref{Hydro-1}}\label{SEC-MOD-SHELL}
 In this section we analyze
the modified stochastic shell model \eqref{Hydro-1}. We are mainly
interested in { the existence and uniqueness of the solution and its qualitative properties}.

\subsection{Resolvability of the modified problem \eqref{Hydro-1}}
Let us first introduce the concept of solution.
\begin{Def}
Let $T>0$ be a real number. An $\mathbb{F}$-adapted process $\bu^R$ taking values in $\h$ is called a solution of Eq. \eqref{Hydro-1}
 if the following conditions are satisfied:
 \begin{enumerate}[(i)]
  \item $\bu^R \in L^2(0,T; \h) $ $\mathbb{P}$-almost surely,
  \item the following equality holds for every $t\in [0,T]$ and $\mathbb{P}$-a.s,
  \begin{equation}
   (\bu^R(t),\phi)=(\xi,\phi)-\int_0^t\left(\langle \kappa \rA\bu^R (s)+\bb^R(\bu^R(s), \bu^R(s)),\phi \rangle\right)ds+
   \langle L(t), \phi \rangle,
  \end{equation}
for any $\phi\in \ve$.
\end{enumerate}
\end{Def}
\begin{prop}\label{MOD-SHELL}
Let the assumptions of Proposition \ref{SHELL} hold. Then for each
$R>0$  the problem \eqref{Hydro-1} has a unique solution $\bu^R$
which has a \cadlag modification in $\h$. Moreover, $\bu^R$ is a
Markov process having the Feller property.
\end{prop}
 \begin{proof}
 Let $\burn$ the solution of the following
\begin{equation}
 \begin{split}\label{Hydro-2}
  & d\burn+[{ \kappa} \rA\burn+\bb^R_n (\burn,\burn)]dt=\sum_{k=1}^n \int_{\mathbb{R}_0} z \beta_k
  e_k\;d\bar{\eta}_k(dz,dt) ,\\
  & \burn(0)=\Pi_n \xi \in \h_n.
 \end{split}
 \end{equation}
 The problem
\eqref{Hydro-2} is a system of SDEs with globally Lipschitz
coefficients. Thus it has a unique c\`adl\`ag solution $\burn$
which is a Markov process taking values in $\h_n$ (see, for
instance, \cite{Albeverio}).

Now the existence and uniqueness of a solution $\bu^R$ can be
established by arguing exactly as in the proof of Proposition
\ref{SHELL}.

Note that thanks to Assumption \ref{A} and Assumption \ref{B} we
can also argue as in \cite{EH+PR} and show that for any $T>0$ we
have
\begin{align}
\lim_{n\to \infty} \EE\lvert \bun^R(T)-\bu^R(T)\rvert^2=0,\nonumber \\
\lim_{n \to \infty}\EE \int_0^T \lvert \rA^\frac12[
\bun^R(s)-\bu^R(s)]\rvert^2 ds=0. \label{CONV-R}
\end{align}

Now it remains to prove that the solution $\bu^R$ to \eqref{Hydro-1} has a
\cadlag modification in $\h$. The argument is very similar to the
idea of the proof of Proposition \ref{SHELL}. Let $\mathfrak{S}$
be the stochastic convolution defined in
\eqref{STOC-CONV-1}-\eqref{STOC-CONV-2}.
\del{ by
\begin{equation*}
\mathfrak{S}(t)=\sum_{k=1}^\infty \mathfrak{S}_k(t) e_k,
\end{equation*}
where each $\mathfrak{S}_k$ is the solution to
\begin{equation*}
d\mathfrak{S}_k(t)=-\lambda_k \mathfrak{S}_n(t) dt+\beta_k
\int_{\mathbb{R}_0} z d\bar{\eta}_k(dz,dt).
\end{equation*}
Since, by assumption, $$ \sum_{k=1}^\infty\left( \lambda_k^{-1}
\beta_k^2 +\beta^4\right)=\sum_{k=1}^\infty
\left(\lambda^{-k(1+2\theta)}+\lambda^{-4\theta k}
\right)<\infty$$ it follows from \cite[Corollary 3.3 and Example
3.4]{SP+JZ-13} that $\mathfrak{S}$ has a  \cadlag modification in
$\h$.}
Arguing as in Appendix \ref{APP-CONV} we can  show that the
following evolution equation
\begin{equation}\label{Eq-8}
\begin{split}
\frac{d}{d t}\bv(t) +{ \kappa}\mathrm{A}\bv(t)+\rho(\lvert
\bv(t)+\mathfrak{S}(t)\rvert^2/R)
\bb(\bv(t)+\mathfrak{S}(t),\bv(t)+\mathfrak{S}(t))
=0, \, t\in [0,T],\\
\bv(0)=\xi\in \h,
\end{split}
\end{equation}
has a unique solution $\bv^R\in C(0,T;\h)\cap L^2(0,T;\ve)$. Now,
the stochastic process $\bu^R$ can be written as $\bu^R=\bv^R +
\mathfrak{S}$ where $\bv^R\in C(0,T;\h)\cap L^2(0,T;\ve)$ is the
unique solution of \eqref{Eq-8}.  Thanks to \cite[Theorem
4.1]{Whitt} the function $$\phi: C([0,T];\h)\times
\mathbb{D}([0,T];\h) \ni (x,y)  \mapsto x+y \in
\mathbb{D}([0,T];\h),$$ is continuous, hence $\bu^R$ has a \cadlag
modification in $\h$.

The proof that $\bu^R$ is a Markov process follows from the
argument in \cite[Section 6]{AMR-09}.

To show that $\bu^R$ has the Feller property we first remark that
for any  $\xi, \zeta\in \h$ we have
\begin{equation*}
\begin{split}
\lvert \bu^R(t,\xi)-\bu^R(t, \zeta)\rvert- C e^{-\lambda_1t}\lve
\xi-\zeta\rve \\
\le  C \int_0^t (t-s)^{-1/2} e^{-\lambda_1(t-s)} \Lve
\bb^R(\bu^R(s,\xi),\bu^R(s,\xi))-\bb^R(\bu^R(s, \zeta),\bu^R(s, \zeta))\Rve_{\ve^\ast}ds
.
\end{split}
\end{equation*}
Since $$\Lve \bb^R(u,u)-\bb^R(v,v) \Rve_{\ve^\ast}\le C_R \lve
u-v\rve,$$ for any $u,v\in \h$ we infer that
\begin{align*}
\lvert \bu^R(t,\xi)-\bu^R(t, \zeta)\rvert\le e^{-\lambda_1t}\lve
\xi-\zeta\rve+C \int_0^t (t-s)^{-1/2} e^{-\lambda_1(t-s)} \lve
\bu^R(s,\xi)-\bu^R(s, \zeta)\rve ds.
\end{align*}
From Gronwall's inequality we deduce that
\begin{equation*}
\lvert \bu^R(t,\xi)-\bu^R(t, \zeta)\rvert\le C(R,T) \lve \xi-\zeta\rve,
\end{equation*}
from which the Feller property of $\bu^R$ easily follows.
 \end{proof}
 \begin{Rem}\label{REM-SECT-4}
  Let $\bu$ be the solution to \eqref{Hydro} and $\{\tau_R; R\in \mathbb{N}\}$ be a sequence of stopping times
  defined by
  $$ \tau_R :=\inf\{t\ge 0; \lvert \bu(t)\rvert\ge R \}.$$
  It is clear that $\bb(\bu(t),\bu(t)):= \bb^R(\bu^R(t),\bu^R(t)$ on $\{t\le 
  \tau_R
  \}$, thus by uniqueness of solution of the system \eqref{Hydro} we infer that $$ \bu=\bu^R \text{ on } \{t\le
  \tau_R
  \}.$$	
 \end{Rem}

 \subsection{Strong feller property of the solution of \eqref{Hydro-2}}
\del{For the proof of \eqref{A} we need to work with the Galerkin
approximation of \eqref{Hydro-1}. That is,}

We have seen in  the proof of Theorem \ref{MOD-SHELL} that the
solution  $\burn$  of the Galerkin approximation (see equation
\eqref{Hydro-2}) generates a Markov semigroup $\CP^R_{t,n}$
defined by
$$\CP^R_{t,n}\Phi(\xi)=\EE[\Phi(\burn(t,\xi))],$$
for any $\Phi\in \mathcal{B}_b(\h_n)$ and $\xi \in \h_n$. 
{ Now, we will show }the smoothing property of $\CP^R_{t,n}$.

 Since
the coefficients of \eqref{Hydro-2} belong to $C^2(\h_n;\h_n)$ the
mapping $\xi_n \ni \h_n \mapsto \burn$ is $C^1$ differentiable and
the derivative $U^R_n(s, x):=\nabla_x \burn(s,\xi)$ in the
direction of $x\in \mathbf{H}_n$ at point $\xi_n\in \h_n$ is the
solution of the linearized equation
\begin{equation}
 \begin{split}\label{Hydro-3}
  & d\Burn(t,x)+[\kappa \rA\Burn(t,x)+\nabla \bb^R_n (\burn(t,\xi),\burn(t,\xi))[\Burn(t,x)]]dt=0,\\
  & \Burn(0)=x.
 \end{split}
 \end{equation}
\begin{lem}
For any $t>0$ and
$\xi\in \h_n$ the system \eqref{Hydro-3} has a unique solution $\Burn$ such that
$\Burn \in C(0,t;\h_n)\cap L^2(0,t;\ve_n)$ . Moreover, for any $R>0$ there exists a constant $C_R>0$ such
that
\begin{equation}\label{DER-GAL}
\sup_{\substack{x\in \h_n,\\\lvert x\rvert\le 1}}\biggl[\EE\lvert
\Burn(s,x)\rvert^2 + \kappa \EE \int_0^t \lvert \mathrm{A}^\frac12
\Burn(s,x)\rvert^2 ds\biggr]\le (1+C_R e^{\frac4{\kappa}t}),\,\,
t>0.
\end{equation}
\end{lem}
\begin{proof}
 For the sake of simplicity we will write $\Burn(\cdot):= \Burn(\cdot,
x)$. We will not dwell on the details of the existence of solution
since it can be proved with standard argument. We will just prove the estimate
\eqref{DER-GAL}.  For this purpose, we start with the following
identity
\begin{align*}
\nabla
\bb^R_n(\burn,\burn)[\Burn]=\rho^\prime_R(\burn)\frac{\langle
\burn, \Burn\rangle}{R \lvert \burn\rvert}
\bb(\burn,\burn)+\rho_R(\burn)[\bb(\burn,\Burn)+\bb(\Burn,\burn)].
\end{align*}
Hence, it follows from Assumption \ref{B}-\eqref{B-a} and the
definition of $\rho_R(\cdot)$ that
\begin{align*}
\Lve \nabla \bb^R_n(\burn,\burn)[\Burn]\Rve_{\ve^\ast}\le
\frac{1}{R}\lvert \rho^\prime(\burn)\rvert\lvert\Burn\rvert \Lve
\bb(\burn,\burn)\Rve_{\ve^\ast}+2 \rho_R(\burn) \Lve
\bb(\burn,\Burn)\Rve_{\ve^\ast}\\
\le \frac{C}{R}\lvert \rho^\prime(\burn)\rvert \,\,
\lvert\Burn\rvert\,\, \lvert \burn\rvert^2+ C \rho_R(\burn) \lvert
\burn \rvert \,\, \lvert \Burn\rvert\\
\le C_R \lvert \Burn\rvert.
\end{align*}
Therefore,
\begin{align}
\lvert \langle \nabla \bb^R_n(\burn,\burn)[\Burn],
\Burn\rangle\rvert\le C_R \lvert \Burn\rvert \,\,\rvert
\mathrm{A}^\frac12 \Burn \rvert\nonumber \\
\le C_R^2 \frac{2}{\kappa} \lvert \Burn \rvert^2+\frac{\kappa}{2}
\lvert \mathrm{A}^\frac12 \Burn\rvert^2.\label{DER-BB-GAL}
\end{align}
Now multiplying \eqref{Hydro-3} by $\Burn(t)$ and plugging
\eqref{DER-BB-GAL} in the resulting equation yields
\begin{align*}
\frac{1}{2}\frac{d}{dt}\lvert \Burn(t)\rvert^2 +\kappa \lvert
\mathrm{A}^\frac12 \Burn(t)\rvert^2 \le C_R^2 \frac{2}{\kappa}
\lvert \Burn(t) \rvert^2+\frac{\kappa}{2} \lvert \mathrm{A}^\frac12
\Burn(t)\rvert^2.
\end{align*}
Thus,
\begin{equation*}
\frac{d}{dt}\lvert \Burn(t)\rvert^2 +\kappa \lvert
\mathrm{A}^\frac12 \Burn(t)\rvert^2 \le C_R^2 \frac{4}{\kappa}
\lvert \Burn(t) \rvert^2,
\end{equation*}
from which along with the Gronwall inequality we infer that
\begin{equation*}
\lvert \Burn(t)\rvert^2 + \kappa \int_0^t \lvert
\mathrm{A}^\frac12 \Burn(s)\rvert^2 ds\le \lvert
\Burn(0)\rvert^2(1+C_R e^{\frac4{\kappa}t}).
\end{equation*}
We easily conclude the proof from this last inequality.
\end{proof}
 We
have the following lemma.
\begin{lem}
Suppose that all the assumptions of Proposition \ref{TSF-MSM} are verified.
Then, for any $R>0$, $t> 0$ there exists a positive constant
$C:=C(t,R)$ such that
\begin{equation}\label{SF-GAL}
\lvert \cprn \Phi(\xi)-\cprn \Phi(\zeta)\rvert< C \Lve \Phi\Rve_\infty
\lvert \xi-\zeta\rvert,
\end{equation}
for any $n\in \mathbb{N}$, $\xi, \zeta\in \h_n$ and $\Phi\in
B_b(\h_n)$.
\end{lem}
\begin{proof}
The idea is to use the estimate for the gradient of  the Markovian
semigroup $\CP^R_{t,n}$. Let $\Phi\in C^2_b(\h_n)$ and $ \nabla_x
\cprn \Phi(\xi)$ be the derivative in the direction of $x\in \h_n$
at a point $\xi\in \h_n$ of $\cprn \Phi(\cdot)$. Notice that when
identifying $\h_n$ with $\mathbb{R}^n$ the linear operator
$\mathrm{A}^\delta$, $\delta \in [0,\infty)$, can be identified
with the diagonal matrix $[\mathrm{A}^\delta_{jk};
j,k=1,\ldots,n]$ defined by
\begin{equation*}
\mathrm{A}^\delta_{jk}=
\begin{cases}
\lambda_j^\delta\text{ if} j=k\\
0 \text{ otherwise}.
\end{cases}
\end{equation*} Thanks to the estimate \eqref{GRAD-EST} in Lemma \ref{Lem-App-B} we have
 \begin{equation*}
\sup_{\substack{n\in \mathbb{N}, \xi, x  \in \h_n\\  \lvert
x\rvert\le 1} } \lvert \nabla_x \cprn \Phi(\xi) \rvert \le C(t)
\left(\sum_{j=1}^\infty \beta_j^{-2} \lambda_j^{-1}
\right)^\frac12 \Lve \Phi\Rve_{\infty} \biggl[\EE\int_0^t \lvert
\mathrm{A}^\frac12 \Burn(s)\rvert^2 ds\biggr]^\frac12,
 \end{equation*}
 where we have used the shorthand notation $\Burn(\cdot):= \Burn(\cdot,
x)$.
 Owing to \eqref{DER-GAL} we obtain the following estimate
\begin{equation*}
 \sup_{\substack{n\in \mathbb{N}, \xi, x  \in \h_n\\  \lvert
x\rvert\le 1} } \lvert \nabla_x \cprn \Phi(\xi) \rvert \le C(t)
\left(\sum_{j=1}^\infty \beta_j^{-2} \lambda_j^{-1}
\right)^\frac12 \Lve \Phi\Rve_{\infty} (1+C_R
e^{\frac4{\kappa}t}).
\end{equation*}
Now we easily derive that  for any $R>0$, $t> 0$
there exists a constant $C:=C(t,R)>0$ such that
\begin{equation*}
\sup_{\substack{n\in \mathbb{N}, \xi, x  \in \h_n\\  \lvert
x\rvert\le 1} } \lvert \nabla_x \cprn \Phi(\xi) \rvert   \le
C(t,R)\Lve \Phi\Rve_\infty ,
\end{equation*}
for any $\Phi \in C^2_b(\h_n)$. Now we easily see that the
estimate \eqref{SF-GAL} holds for $\Phi\in C^2_b(\h_n)$. Owing to
the equivalence lemma \cite[Lemma 2.2]{SP+JZ} it follows that
\eqref{SF-GAL} also holds for $\Phi\in B_b(\h_n)$, and this
completes the proof of our claim.
\end{proof}

\subsection{Strong Feller property of the solution to \eqref{Hydro-1}} In
this section we will prove that for any $R>0$ the semigroup
$\CP^R_t$ associated to the solution $\bu^R$ of the modified
problem \eqref{Hydro-1} has the strong Feller property.
\begin{prop}
Suppose that all the assumptions of Proposition
\ref{TSF-MSM} are satisfied. Then, for any $R>0$, $t> 0$ there exists a
positive constant $C:=C(t,R)$ such that
\begin{equation}\label{SF-MSM}
\lvert \CP^R_t \Phi(\xi)-\CP^R_t  \Phi(\zeta)\rvert< C \Lve \Phi
\Rve_\infty \lvert \xi-\zeta\rvert,
\end{equation}
for any $\xi, \zeta\in \h$ and $\Phi\in B_b(\h)$.
\end{prop}
\begin{proof}
Since, by \eqref{CONV-GAL}, $\burn$ converges to $\bu^R$ strongly
in $L^2(0,t;H)$ $\mathbb{P}$-a.s. we can infer the existence of a
subsequence $n_k$ such that
\begin{equation*}
\bu^R_{n_k}\rightarrow \bu^R \,\,\,\, dt\times
d\mathbb{P}-\text{almost everywhere}.
\end{equation*}
Thanks to this convergence, the continuity and the boundedness of
$\Phi$ we can derive from the Lebesgue Dominated Convergence
Theorem that as $n_k\rightarrow 0$
\begin{equation*}
\EE \int_0^t \lvert
\Phi(\bu^R_{n_k}(t,\xi))-\Phi(\bu^R(t,\xi))\rvert ds\rightarrow 0.
\end{equation*}
Hence there exists a subsequence, denoted again by $n_k$, such
that
\begin{equation*}
\EE [\Phi(\bu^R_{n_k}(s,\xi))] \rightarrow \EE [\Phi(\bu^R
(s,\xi))]
\end{equation*}
for almost all $s\in [0, t]$. Hence thanks to this last
convergence and \eqref{SF-GAL} there exists $I_0\subset (0,t]$
with $\mathrm{Leb}(I_0)=0$ such that for any $R>0$ and $s\in I_0$
we have
\begin{equation}\label{SF-AE}
\lvert \CP^R_s\Phi(\xi)-\CP^R_s\Phi(\zeta)\rvert\le C_{R,s} \Lve
\Phi\Rve_\infty \lvert\xi-\zeta\rvert,
\end{equation}
for any $\xi, \zeta\in \h$ and $\Phi \in C^2_b(\h)$. Since, by
Proposition \ref{MOD-SHELL}, $\bu^R$ is c\`adl\`ag in $\h$, the function $s\mapsto \CP^R_s \Phi(\xi)-\CP^R_s\Phi(\zeta)$, for any $\xi, \zeta\in \h$ and $\Phi \in C^2_b(\h)$, is also c\`adl\`ag and the estimate
\eqref{SF-AE} holds for all $s\in (0,t]$. This ends the proof of
our claim.
\end{proof}
\appendix
\section{Bismut-Elworthy-Li formula}\label{SEC-BEL}
In this first appendix we give and prove a Bismut-Elworthy-Li type
formula for stochastic differential equations driven by pure jump
noise. The proof is mainly a modification of \cite[Proof of
Theorem 1]{Takeuchi}. We will also follow closely the notation in
\cite{Takeuchi}.

Let $\eta:=(\eta_1, \ldots, \eta_n)$ be a Poisson random measure
on $\mathbb{R}^n$, where $\eta_1, \ldots, \eta_n$ are
$n$-independent Poisson random measure with L\'evy measures
$\nu_1, \ldots, \nu_n$ on
$\mathbb{R}_0:=\mathbb{R}\backslash\{0\}$. The L\'evy measure
of $\eta$ is denoted by
$\nu(dz):=(\nu_1(dz_1), \ldots, \nu_n(dz_n))$ . We use the symbol $$(\widehat{\eta}_1(z_1,t), \ldots,
\widehat{\eta}_n(z_n,t):=(\nu_1(dz_1)dt, \ldots, \nu_n(dz_n)dt),$$ to
denote the compensator of $\eta$. The symbol
$\tilde{\eta}=(\tilde{\eta_1},\ldots, \tilde{\eta_n})$ describes the
compensated Poisson random measures associated to $\eta$. To
shorten notation we will use the following shorthand notations
$d\eta(z,t):=\eta(dz,dt)$, $d \tilde{\eta}(z,t):=\tilde{\eta}(dz, dt)$ and
$d\nu(z)dt:=\nu(dz)dt$. We will also use the notation
$d\bar{\eta}(z,t):=(d\bar{\eta}_1(z_1,t),\ldots, d\bar{\eta}_n(z_n,t)$ where
$$d\bar{\eta}_j(z_j,t)=\mathds{1}_{\lvert z_j\rvert\le 1} d\tilde{\eta}_j(z_j,t)+
\mathds{1}_{\lvert z_j\rvert>1} d\eta_j(z_j,t).$$

For this appendix we impose the following sets of conditions.
\begin{assum}\label{Lev-Meas-bis}
\begin{enumerate}
\item \label{LMbi} For each $j$ there exists a $C^1$ function
$g_j: \mathbb{R} \rightarrow [0,\infty)$ such that
$\nu_j(dz_j)=g_j(z_j)dz_j$.

\item \label{LMbii} As $\lvert z\rvert \rightarrow \infty$ we have
$z^2 g_j(z)\rightarrow 0$ for each $j$. Also
$$ \int_{\mathbb{R}_0} \Big[\lvert z\rvert^q \mathds{1}_{(\lvert
z\rvert\le 1)} + \lvert z\rvert^q \mathds{1}_{(\lvert z\rvert>1
)}\Big]\nu_j (dz)<\infty,
$$ for any $q\ge 1$ and $j\in \{1,\ldots, n\}$.

\item \label{LMbiii}Furthermore, there exists a constant
$\alpha>0$ such that for any  $j\in \{1,\ldots,n\}$ and $y\in
\mathbb{R}$
$$ \liminf_{\eps\searrow 0}\eps^\alpha \int_{\mathbb{R}_0} (\lvert
z\cdot y/\eps\rvert^2\wedge 1)\nu_j(dz)>0.$$
\end{enumerate}
\end{assum}
\begin{assum}\label{DER}
Let $\alpha:\mathbb{R}^n \rightarrow \mathbb{R}^n$ be a nonlinear
map such that $\alpha(\cdot)$ belongs to $C^2_b(\mathbb{R}^n,
\mathbb{R}^n)$.
\end{assum}
Let $\gamma$ be the $n\times n$  diagonal matrix given by
\begin{equation*}
\gamma_{ij}(z)=
\begin{cases}
\beta_i z_i \text{ if } i=j,\\
0 \text{ otherwise,}
\end{cases}
\end{equation*}
where $\{\beta_i; i=1,\ldots, n\}$ is a family of positive
numbers.

 Let
$\bx(t,x):=(\bxi{1}(t,x), \ldots, \bxi{n}(t,x))$ be the unique
solution to the system of $n$ SDEs given by
\begin{equation*}
d\bx(t)=\mathbf{\alpha}(\bx(t)) dt + \int_{\mathbb{R}_0^n}
\gamma(z) d\bar{\eta}(z,t), \,\, \bx(0)=x\in \mathbb{R}^n
\end{equation*}
that is
\begin{equation}\label{SYS-P}
d\bxi{i}(t)=\alpha^{(i)}(\bx(t)) dt+\sum_{j=1}^n
\int_{\mathbb{R}_0}\gamma_{ij}(z_j) d\bar{\eta}_j(z_j,t),\,\,
\bxi{i}(0)=x_i \,\, i=1,\ldots,n.
\end{equation}
Note that for any $x\in \err^n$  the process $\bx(\cdot,x)$ is a Markov process. Thanks to Assumption \ref{DER}, it is proved in \cite{Takeuchi} that the map $\mathbb{R}^n\ni x \mapsto
\bx(t)$ has a $C^1$-modification and its Jacobi matrix
$U(t):=[U_{kj}(t); k,j\in \{1,\ldots,n\}]=[\frac{\partial
\bxi{j}(t)}{\partial x_k}; k,j\in \{1,\ldots,n\}]$ satisfies
\begin{equation*}
\frac{d }{d t}U(t) =\nabla_x \alpha(\bx(t,x))U(t), \,\,\,\,
U(0)=I_n
\end{equation*}

 Let $\Lambda(s,z)$ be the matrix defined by
\begin{equation*}
\Lambda_{kj}(s,z)=z_j^2 g_j(z_j) \beta^{-1}_j \frac{\partial
\bxi{j}(s,x)}{\partial x_k},
\end{equation*}
and $J(t):=(\ji{1}(t), \ldots, \ji{n}(t))$ be the vector defined by
\begin{equation*}
\ji{k}(t)=\sum_{j=1}^n \int_0^t \int_{\mathbb{R}_0}\frac1
{g_j(z_j)} \frac{\partial \Lambda_{kj}(s,z)}{\partial z_j}
d\tilde{\eta}_j(z_j,s).
\end{equation*}
For $t>0$ we set
\begin{equation*}
\ki{k}(t)=-2\sum_{j=1}^n \int_0^t \int_{\mathbb{R}_0} z_j
\beta^{-1}_j \frac{\partial \bxi{j}(s,x)}{\partial x_k} d\eta_j(z_j,s),
\end{equation*}
 and
  $$ \mathcal{A}(t)=\sum_{j=1}^n \int_0^t \int_{\mathbb{R}_0}
z_j^2 d\eta_j(z_j,s).$$
\begin{lem}\label{BEL}
 Let Assumption \ref{Lev-Meas-bis} and Assumption \ref{DER}
 hold. Then
\begin{align*}
\nabla_{x_k} \EE[ \Phi(\bx(t,x)]=
\EE\left[\Phi(\bx(t,x))\frac{\ki{k}(t)}{[\mathcal{A}(t)]^2}\right]-\EE\left[\Phi(\bx(t,x))
\frac{\ji{k}(t)}{\mathcal{A}(t)}\right],
\end{align*}
for any $\Phi\in C^2_b(\mathbb{R}^n)$, $x\in \mathbb{R}^n$ and
$k\in \{1,\ldots,n\}$.
\end{lem}
\begin{proof}
Our proof is mainly based on the arguments in \cite[Section 4,
Proof of Theorem 1]{Takeuchi} to which we refer for the omitted
details. Here we only dwell on the parts where our idea and the
arguments in \cite{Takeuchi} differ.

 Let $u(s,x):=
\EE[\Phi(\bx(t-s,x))\lvert \bx(0)=x]$ for any  $s\in [0,t]$ and
$\Phi \in C^2_b(\mathbb{R}_n)$.  Let $\gi{j}$ be the $j$-th column of
the matrix $\gamma$ and for any $y\in \mathbb{R}^n$ let
\begin{equation*}
\biz{j}f(y):=f(y+\gi{j}(z))-f(y).
\end{equation*}
Arguing as in \cite[Lemma 4.1]{Takeuchi} and \cite[Lemma 4.2]{Takeuchi} respectively, we can show that
\begin{align}\label{erste_zeile}
\Phi(\bx(t,x))=\EE[ \Phi(\bx(t,x))]+\sum_{j=1}^n \int_0^t
\int_{\mathbb{R}_0} \biz{j}
u(s, \bx(s-,x)) d\tilde{\eta}_j(z_j,s),
\end{align}
and
\begin{align}\label{zweit_zeile}
\nabla_{x_k} \Phi(\bx(t,x))=\EE [\nabla_{x_k}
\Phi(\bx(t,x))]+\sum_{j=1}^n \int_0^t \int_{\mathbb{R}_0}
\nabla_{x_k} \biz{j} u(s, \bx(s-,x)) d\tilde{\eta}_j(z_j,s),
\end{align}
for any $k\in \{1, \ldots, n\}$.

Hereafter we fix $k, j\in \{1,\ldots, n\}$.
Since, by Assumption \ref{Lev-Meas-bis}-\eqref{LMbii}, $ \int_{\mathbb{R}_0} z^2 \nu_j(dz_j)<\infty$ for any $j=1,\ldots,n$, we can use the same argument as in
\cite[Proof of Lemma 4.6]{Takeuchi} to prove that
\DEQSZ\label{mmmmh}
\\
\nonumber
 \EE \biggl[\sum_{j=1}^n \int_0^t \int_{\mathbb{R}_0}  \biz{j} u(s,
\bx(s-,x)) d\tilde{\eta}_j(z_j,s) \times \sum_{j=1}^n \int_0^t
\int_{\mathbb{R}_0} z_j^2 d\tilde{\eta}_j\biggr]=\sum_{j=1}^n
\EE\biggl[\Phi(\bx(t,x)) \int_0^t \int_{\mathbb{R}_0}
z_j^2 d\tilde{\eta}_j(z_j,s)\biggr].
\EEQSZ
In fact, if we replace  the right hand side (RHS) $\Phi(\bx(t,x))$ by the RHS of \eqref{erste_zeile} and
take into account that 
$$\EE\lk[ \, \EE[\Phi(\bx(t,x))] \, \int_0^t \int_{\mathbb{R}_0}
z_j^2 d\tilde{\eta}_j(z_j,s)\rk]=0,
$$
then \eqref{mmmmh} follows from the It\^o formula.
Similarly, we can use equation \eqref{zweit_zeile}, to show first
\begin{equation}\label{eq-1}
\begin{split}
 \EE \biggl[\sum_{j=1}^n \int_0^t \int_{\mathbb{R}_0} \nabla_{x_k} \biz{j} u(s,
 \bx(s-,x)) d\tilde{\eta}_j(z_j,s) \times \sum_{j=1}^n \int_0^t
 \int_{\mathbb{R}_0} z_j^2 d\tilde{\eta}_j(z_j,s)\biggr]
 \\ =\sum_{j=1}^n
 \EE\biggl[\nabla_{x_k} \Phi(\bx(t,x)) \int_0^t \int_{\mathbb{R}_0}
 z_j^2 d\tilde{\eta}_j(z_j,s)\biggr],
\end{split}
\end{equation}
and, secondly, by  It\^o's formula
\begin{equation} \label{eq-3}
\begin{split}
\EE \biggl[\sum_{j=1}^n\int_0^t \int_{\mathbb{R}_0} \nabla_{x_k}
\biz{j} u(s, \bx(s-,x)) d\tilde{\eta}_j(z_j,s) \times \sum_{j=1}^n
\int_0^t \int_{\mathbb{R}_0} z_j^2
d\tilde{\eta}_j(z_j,s)\biggr]\\ =\sum_{j=1}^n \EE \biggl[\int_0^t\int_{\mathbb{R}_0}
\nabla_{x_k} \biz{j}u(s,\bx(s-,x)) z_j^2 d\nu_j(z_j)ds\biggr].
\end{split}
\end{equation}
Hence, plugging this last identity, i.e.\ \eqref{eq-3}, in \eqref{eq-1}  yields
\DEQSZ
\label{eq-2}
\\
\nonumber
\EE\biggl[\nabla_{x_k} \Phi(\bx(t,x)) \times \sum_{j=1}^n \int_0^t
\int_{\mathbb{R}_0} z_j^2 d\tilde{\eta}_j(z_j,s)\biggr]=\sum_{j=1}^n\EE
\biggl[\int_0^t\int_{\mathbb{R}_0} \nabla_{x_k} \biz{j}u(s,\bx(s-,x)) z_j^2
d\nu_j(z_j)ds\biggr].
\EEQSZ
In the other hand, since
$u(s,\bx(s,x))=\EE[\Phi(\bx(t-s,y))\lvert y=X(s,x)]$, we easily deduce  from the Markov and the tower property of mathematical expectation that
$$ \EE u(s,\bx(s,x))=\EE\lk[  \EE[\Phi(\bx(t-s,y))\lvert  y=X(s,x)]\rk] =\EE\lk[\Phi(\bx(t,x)) \rk] .$$ Thus, we infer from Fubini's theorem 
that
\DEQS 
\lqq{ \sum_{j=1}^n \EE\int_{0}^{t} \int_{\mathbb{R}_0}u(s, \bx(s,x)) z_j^2 d\nu_j(z_j)ds}
\\
&=&  \sum_{j=1}^n \int_{0}^{t} \int_{\mathbb{R}_0}\EE\lk[ \, \EE[\Phi(\bx(t-s,y))\lvert y=\bx(s,x)] \rk]\, z_j^2 d\nu_j(z_j)ds
\\
&=& \,  \sum_{j=1}^n \int_{0}^{t} \int_{\mathbb{R}_0}\EE \Phi(\bx(t,x)) z_j^2 d\nu_j(z_j)ds\lvert 
\\
&
=&\EE\biggl[\Phi(\bx(t,x)) \times \sum_{j=1}^n \int_0^t
\int_{\mathbb{R}_0} z_j^2 d\nu_j(z_j)ds\biggr].
\EEQS
In the very same way we get 
\begin{equation*}
\EE\biggl[\nabla_{x_k} \Phi(\bx(t,x))\times \sum_{j=1}^n \int_0^t
\int_{\mathbb{R}_0} z_j^2 d\nu_j(z_j)ds \biggr]=\sum_{j=1}^n
\EE\biggl[\int_0^t\int_{\mathbb{R}_0} \nabla_{x_k} u(s,
\bx(s-,x))z_j^2 d\nu_j(z_j)ds\biggr].
\end{equation*}
Now, observe that we have 
\DEQSZ\label{obenhhh}
 \nabla_{x_k}\biz{j}u(s,\bx(s-,x))&= \nabla_{x_k}u(s,\bx(s-,x)+\gamma^{(j)}(z))-\nabla_{x_k}u(s,\bx(s-,x)),
\EEQSZ
and
\DEQSZ\label{oben1}
 \int_0^t\int_{\mathbb{R}_0}z^2_j d\eta_j(z_j,s)&= \int_0^t\int_{\mathbb{R}_0} z^2_j d\tilde{\eta}_j(z_j,s)+ \int_0^t\int_{\mathbb{R}_0} z^2_j d\nu_j(z_j)ds,
\EEQSZ
Therefore, using \eqref{oben1} in one hand yields
\DEQS
\lqq{
\EE\biggl[\nabla_{x_k} \Phi(\bx(t,x)) \times \sum_{j=1}^n \int_0^t
\int_{\mathbb{R}_0} z_j^2 d{\eta}_j(z_j,s)\biggr]}
\\
&=&
\EE\biggl[\nabla_{x_k} \Phi(\bx(t,x)) \times \sum_{j=1}^n \int_0^t
\int_{\mathbb{R}_0} z_j^2 d\tilde{\eta}_j(z_j,s)\biggr]+\EE\biggl[\nabla_{x_k} \Phi(\bx(t,x)) \times \sum_{j=1}^n \int_0^t
\int_{\mathbb{R}_0} z_j^2 d{\nu}_j(z_j,s)\biggr],
\EEQS
and in the other hand, using \eqref{obenhhh}
we derive from \eqref{eq-2}  that
\DEQS
\lqq{
\EE\biggl[\nabla_{x_k} \Phi(\bx(t,x)) \times \sum_{j=1}^n \int_0^t
\int_{\mathbb{R}_0} z_j^2 d{\eta}_j(z_j,s)\biggr]}
\\
&=& \sum_{j=1}^n\EE
\biggl[\int_0^t\int_{\mathbb{R}_0} \nabla_{x_k} \biz{j}u(s,\bx(s-,x)) z_j^2
d\nu_j(z_j)ds\biggr]+\EE\biggl[\nabla_{x_k} \Phi(\bx(t,x)) \times \sum_{j=1}^n \int_0^t
\int_{\mathbb{R}_0} z_j^2 d{\nu}_j(z_j,s)\biggr]
\\
&=& \sum_{j=1}^n
\EE\biggl[\int_0^t \int_{\mathbb{R}_0} \nabla_{x_k} u(s,
\bx(s-,x)+\gi{j}(z)) z_j^2 d\nu_j(z_j)ds\biggr].
\EEQS
That is,
\DEQSZ\label{eq-333}
\lqq{ \EE\biggl[\nabla_{x_k} \Phi(\bx(t,x)) \times \sum_{j=1}^n \int_0^t
\int_{\mathbb{R}_0} z_j^2 d{\eta}_j(z_j,s)\biggr]} &&
\\
&=&\sum_{j=1}^n
\EE\biggl[\int_0^t \int_{\mathbb{R}_0} \nabla_{x_k} u(s,
\bx(s-,x)+\gi{j}(z)) z_j^2 d\nu_j(z_j)ds\biggr].
\nonumber
\EEQSZ
Next, by the same argument as used to show formula \eqref{mmmmh} we obtain,
(see also  \cite[Proof of Lemma 4.6]{Takeuchi})
\DEQS
\lqq{ \EE[\Phi(\bx(t,x)) )\ji{k}(t)]}
&&
\\ &=&\EE\biggl[\sum_{j=1}^n \int_0^t
\int_{\mathbb{R}_0} \biz{j} u(s,\bx(s-,x)) d\tilde{\eta}_j(z_j,s) \times
\sum_{j=1}^n \int_0^t \int_{\mathbb{R}_0} \frac1 {g_j(z_j)}
\frac{\partial
\Lambda_{kj}(s,z)}{\partial z_j} d\tilde{\eta}_j(z_j,s)\biggr].
\EEQS
Continuing, and using the definition of $g_j$ we get
\DEQS
\lqq{ \EE[\Phi(\bx(t,x)) )\ji{k}(t)]}
&&
\\
&=&\sum_{j=1}^n \EE\biggl[\int_0^t \int_{\mathbb{R}_0} \biz{j} u(s,
\bx(s-,x)) \frac1 {g_j(z_j)} \frac{\partial
\Lambda_{kj}(s,z)}{\partial z_j}
d\nu_j(dz_j)ds\biggr]\\
&=&\sum_{j=1}^n \EE\biggl[ \int_0^t \int_{\mathbb{R}_0} \biz{j} u(s,
\bx(s-,x)) \frac{\partial \Lambda_{kj}(s,z)}{\partial z_j}
dz_jds\biggr],
\EEQS
%
By integration-by-parts, using the Assumption
\ref{Lev-Meas-bis}-\eqref{LMbii}, and recalling the definition of $\Lambda_{kj}$ we derive that
\begin{align}
\EE[\Phi(\bx(t,x))\ji{k}(t)]= -\sum_{j=1}^n \EE\biggl[ \int_0^t
\int_{\mathbb{R}_0} \frac{\partial}{\partial z_j} \biz{j}
u(s,\bx(s-,x)) \Lambda_{kj}
dz_j ds\biggr]\nonumber \\
=-\sum_{j=1}^n \EE\biggl[\int_0^t \int_{\mathbb{R}_0}
\frac{\partial}{\partial z_j} \biz{j} u(s,\bx(s-,x)) z_j^2
\beta^{-1}_j \frac{\partial
\bxi{j}(s-,x)}{\partial x_k} g_j(z_j) dz_jds\biggr]\nonumber \\
=-\sum_{j=1}^n \EE\biggl[\int_0^t \int_{\mathbb{R}_0}
\frac{\partial}{\partial z_j} \biz{j} u(s,\bx(s-,x)) z_j^2
\beta^{-1}_j \frac{\partial \bxi{j}(s-,x)}{\partial x_k}
d\nu_j(z_j)ds\biggr].\label{lastline}
\end{align}
 We have the following partial derivatives identities
\begin{align*}
\frac{\partial}{\partial z_j} \biz{j}u(s, \bx(s-,x))&= \frac{\partial }{\partial y_j}  u^{(j)}(s,y)\Big|_{y= \bx(s-,x)+\gamma^{(j)}(z)}
\; \frac{\partial \gamma^ {(j)}}{\partial z_j}= \beta_j\;  \frac{\partial }{\partial y_j}  u^{(j)}(s,y)\Big|_{y= \bx(s-,x)+\gamma^{(j)}(z)}
\end{align*}
and 
{
\begin{align*}
\nabla_{x_k}  u^{(j)}(s,\bx(s-,x)+\gamma^{(j)}(z))
&=\frac{\partial u^{(j)}(s,\bx(s-,x)+\gamma^{(j)}(z))
  }{\partial x_k}
\\   &=  \frac{\partial  }{\partial y_j} u^{(j)}(s,y)\Big|_{y=\bx(s-,x)+\gamma^{(j)}(z)}
  \;
  \frac{\partial \bxi{j}(s-,x)}{\partial x_k}.
\end{align*}
}
Hence,
\begin{align*}
\beta_j^{-1} \frac{\partial}{\partial z_j} \biz{j}u(s, \bx(s-,x)) \frac{\partial}{\partial x_k}
\bxi{j}(s,x)&= \nabla_{x_k}  u^{(j)}(s,\bx(s-,x)+\gamma^{(j)}(z)).
\end{align*}
Using this identity in  \eqref{lastline}
implies
\begin{align*}
\EE[\Phi(\bx(t,x) \ji{k}(t)]= -\sum_{j=1}^n \EE\biggl[\int_0^t
\int_{\mathbb{R}_0} \nabla_{x_k} 
u(s,\bx(s-,x)+\gamma^{(j)}(z)) z_j^2
d\nu_j(z_j)ds\biggr],
\end{align*}
from which along with the identity
\eqref{eq-333} 
we derive that
\begin{align}
\EE[\Phi(\bx(t,x) \ji{k}(t)]= -\EE\Big[\nabla_{x_k} \Phi(\bx(t,x))\times \sum_{j=1}^n \int_0^t
\int_{\mathbb{R}_0} z_j^2 d\eta_j(z_j,s)\Big].\label{eq-4}
\end{align}
The next task is to get rid of $\int_0^t
\int_{\mathbb{R}_0} z_j^2 d\eta_j(z_j,s)$ and take $\nabla_{x_k}$ out of the mathematical expectation in the formula above.
To this end, for any $\lambda>0$ let
\DEQS
\zl{t} &=&\exp\biggl(-\sum_{j=1}^n \int_0^t \int_{\mathbb{R}_0}
\lambda z^2_j d\eta_j(z_j,s) -\sum_{j=1}^n \int_0^t \int_{\mathbb{R}_0}
\left(e^{-\lambda z^2_j}-1\right) d\nu_j(z_j)ds\biggr)
\\
&=& \exp\biggl(-\mathcal{A}(t) -\sum_{j=1}^n \int_0^t \int_{\mathbb{R}_0}
\left(e^{-\lambda z^2_j}-1\right) d\nu_j(z_j)ds\biggr).
\EEQS
Thanks to \cite[Theorem 1.4 and Remark 1.2]{Ishikawa} the process $[0,\infty)\ni t\mapsto\zl{t}$, 
is a martingale and one can define a new probability
$\mathbb{P}^\lambda$ on $(\Omega, \mathcal{F})$ such that $$
\frac{d
\mathbb{P}^\lambda}{d\mathbb{P}}{\biggl\vert_{\mathcal{F}_t}}=\zl{t}.$$
Moreover, under $\mathbb{P}^\lambda$ the random measure $\eta_j$,
$j\in \{1,\ldots, n\}$, is a Poisson random measure with L\'evy
measure
$$ \nu^\lambda_j(dz_j)=e^{-\lambda z_j^2}g_j(z_j) dz_j=e^{-\lambda
z_j^2} \nu_j(dz_j),$$ and $d\ekl_j:=d\eta_j-d\efl_j$, where
$d\efl_j(z_j)dt=\nu^\lambda_j(dz_j)dt$, is a martingale measure. The
solution under $\mathbb{P}$ to the system \eqref{SYS-P}  has,
under $\mathbb{P}^\lambda$, the same law as the solution
$\bx:=(\bxi{1},\ldots, \bxi{n})$  of the following system
\begin{equation}\label{SYS-LP}
\begin{split}
d\bxi{i}(t,x)=&\alpha^{(i)}(X(t,x))dt +\sum_{j=1}^n
\int_{\mathbb{R}_0} \gamma_{ij}(z) (e^{-\lambda
z_j^2}-1)d\nu_j(z_j)dt+ \sum_{j=1}^n
\int_{\mathbb{R}_0} \gamma_{ij}(z) d\ebl_j(z_j,t),\\ \bxi{i}(0)=& x_i,\qquad
i= 1,\ldots, n.
\end{split}
\end{equation}
Arguing as in the proof of \eqref{eq-4} we can show that
\begin{equation*}
\nabla_{x_k}\EE^\lambda \Big[\Phi(\bx(t,x) \times \sum_{j=1}^n
\int_0^t \int_{\mathbb{R}_0} z_j^2 d\eta_j(z_j,s)\Big]=-\EE^\lambda
[\Phi(\bx(t,x)\ji{k}_\lambda(t)],
\end{equation*}
where
\begin{equation*}
\ji{k}_\lambda(t):=\sum_{j=1}^n \int_0^t \int_{\mathbb{R}_0}
\frac{1}{e^{-\lambda z_j^2} g_j(z_j)}
\frac{\partial\Lambda^\lambda_{kj}(s,z)}{\partial z_j}  d\ekl_j(z_j,s),
\end{equation*}
and $$ \Lambda^\lambda_{kj}=e^{-\lambda z_j^2} z_j^2 \beta^{-1}_j
\frac{\partial \bxi{j}(s,x)}{\partial x_k}.$$ For the sake of
simplicity, let us set
\begin{align*}
N_\lambda(t)=\sum_{j=1}^n \int_0^t \int_{\mathbb{R}_0}
(e^{-\lambda z_j^2}-1) d\nu_j(z_j)ds.
\end{align*}
Note that
\begin{align*}
\int_0^\infty e^{N_\lambda(t)} \EE^\lambda\left[\Phi(\bx(t,x))
\ji{k}_\lambda(t)\right] d\lambda=\int_0^\infty e^{N_\lambda(t)}
\EE \left[\zl{t} \Phi(\bx(t,x)) \ji{k}_\lambda(t)\right]d\lambda\\
=\int_0^\infty \EE\left[e^{-\lambda \mathcal{A}(t)} \Phi(\bx(t,x))
\ji{k}_\lambda(t)\right] d\lambda.
\end{align*}
 Since $ d\ekl_j(z_j,s)=d\tilde{\eta}_j(z_j,s)+(1-e^{\lambda
z_j^2}) \nu_j(z_j)ds$, we can use integration-by-parts and
Assumption \ref{Lev-Meas}-\eqref{LMbii} to show that $
\ji{k}_\lambda(t)=\ji{k}(t)-\lambda \ki{k}(t).$ Hence
\begin{align*}
\int_0^\infty e^{N_\lambda(t)} \EE^\lambda\left [\Phi(\bx(t,x))
\ji{k}_\lambda(t)\right] d\lambda= \int_0^\infty
\EE\left[e^{-\lambda \mathcal{A}(t)} \Phi(\bx(t,x))
(\ji{k}(t)-\lambda \ki{k}(t)) \right]d\lambda.
\end{align*}
Thanks to Fubini's theorem we infer that
\begin{align*}
\int_0^\infty e^{N_\lambda(t)} \EE^\lambda\left [\Phi(\bx(t,x))
\ji{k}_\lambda(t)\right] d\lambda= \EE \biggl[\int_0^\infty
e^{-\lambda \mathcal{A}(t)} \Phi(\bx(t,x)) (\ji{k}(t)-\lambda
\ki{k}(t))
d\lambda\biggr]\\
=\EE\left[\Phi(\bx(t,x))
\frac{\ji{k}(t)}{\mathcal{A}(t)}\right]-\EE\left[\Phi(\bx(t,x))\frac{\ki{k}(t)}{[\mathcal{A}(t)]^2}\right].
\end{align*}
But, as in \cite[Proof of Theorem 1]{Takeuchi} the following identity holds
\begin{equation*}
\nabla_{x_k} \EE [\Phi(\bx(t,x)]=\nabla_{x_k}
\EE\left[\Phi(\bx(t,x))
\frac{\mathcal{A}(t)}{\mathcal{A}(t)}\right]=\int_0^\infty
\nabla_{x_k} \EE[\Phi(\bx(t,x)) \mathcal{A}(t) \zl{t}]e^{N_\lambda(t)}
d\lambda,
\end{equation*}
from which we infer that
\begin{align*}
\nabla_{x_k} \EE [\Phi(\bx(t,x)]=\int_0^\infty \nabla_{x_k}
\EE^\lambda\left[\Phi(\bx(t,x))\mathcal{A}(t)\right]
e^{N_\lambda(t)}
d\lambda\\
=-\int_0^\infty e^{N_\lambda(t)} \EE^\lambda[\Phi(\bx(t,x))
\ji{k}_\lambda(t)] d\lambda.
\end{align*}
Therefore,
\begin{align*}
\nabla_{x_k} \EE[
\Phi(\bx(t,x)]=\EE\left[\Phi(\bx(t,x))\frac{\ki{k}(t)}{[\mathcal{A}(t)]^2}\right]-\EE\left[\Phi(\bx(t,x))
\frac{\ji{k}(t)}{\mathcal{A}(t)}\right],
\end{align*}
for any $k\in \{1,\ldots,n\}$. This completes the proof of our
lemma.
\end{proof}
\section{Estimates of $\nabla_x
\EE[\Phi(\bx(t,x)]$}\label{SEC-GRAD-EST} In this section we will
derive estimates for the gradient of the Markov semigroup
$\EE[\Phi(\bx(t,x)]$. Let $\{\lambda_j; j=1,2,\ldots\}$ be a
sequence of positive numbers, $\delta \in [0,\frac12]$ and
$\mathrm{A}^\delta$ be the matrix defined by
\begin{equation*}
\mathrm{A}^\delta_{jk}=
\begin{cases}
\lambda_j^\delta \text{ if } j=k\\
0 \text{ otherwise}.
\end{cases}
\end{equation*}
\begin{lem}\label{Lem-App-B}
Assume that Assumption \ref{Lev-Meas} and Assumption \ref{DER} are verified instead. Then,
for any $t>0$ there exists a constant $C=C(t)$ such that
\begin{equation}\label{GRAD-EST}
\lvert \nabla_x \EE [\Phi(\bx(t,x))]\rvert\le C(t)
\left(\sum_{j=1}^n \beta_j^{-2}
\lambda_j^{-2\delta}\right)^\frac12
 \lvert \Phi\rvert_\infty \biggl[\EE \int_0^t
\lvert \mathrm{A}^\delta \nabla_x \bx(s,x) \rvert^2 ds
\biggr]^\frac12,
\end{equation}
 for any
$x\in \mathbb{R}^n$ and $\Phi  \in C^2_b(\mathbb{R}^n)$.
\end{lem}
\begin{proof}
 For $p=1,2$ and $\delta \in
[0,\frac12]$ let
$$ C_p(t)=\EE\left( \frac{1}{[\mathcal{A}(t)]^{2p}}\right),$$
where
  $$ \mathcal{A}(t)=\sum_{j=1}^n \int_0^t \int_{\mathbb{R}_0}
z_j^2 d\eta_j(z_j,s),$$
and $$C_n=\sum_{j=1}^n \beta_j^{-2} \lambda_j^{-2\delta},
$$

Before we proceed to the proof we should note that owing to Assumption \ref{Lev-Meas} and Assumption \ref{DER} we have that
\begin{equation}\label{rem-kely}
\EE\Big[\sup_{s\in [0,t]} \left(\lvert \bx(s,x) \rvert^2 + \lvert A^\delta \nabla _x \bx(s,x)\rvert^2 \right)\Big]<\infty.
\end{equation}

 From the identity in Lemma \ref{BEL} we drive that
\begin{align}
\lvert \nabla_x \EE[\Phi(\bx(t,x))]\rvert^2 \le \lve
\Phi\rve_\infty^2  \biggl( C_2(t) \sum_{k=1}^n \EE [\ki{k}(t)]^2
+C_1(t) \sum_{k=1}^n \EE [\ji{k}(t)]^2 \biggr).\label{COR-BEL}
\end{align}

Let us first estimate the term  $\sum_{k=1}\EE \lvert \ji{k}(t)\rvert^2.$
From discrete Cauchy-Schwarz inequality and It\^o's isometry we
derive that
\begin{align*}
\EE \lvert \ji{k}(t)\rvert^2\le C \sum_{j=1}^n
\beta_j^{-2}\lambda_j^{-2\delta} \sum_{j=1}
\EE\biggl(\int_0^t\int_{\mathbb{R}_0} \lambda_j^{\delta}
\frac{1}{g_j(z_j)} \frac{\partial}{\partial z_j} [z_j^2g_j(z_j)]
\frac{\partial \bxi{j}(s)}{\partial x_k} d\tilde{\eta}_j(z_j,s)
\biggr)^2\\
\le C C_n \sum_{j=1}^n \EE\int_0^t \int_{\mathbb{R}_0}
\left[\frac{1}{g_j(z_j)} \frac{\partial}{\partial z_j}
(z_j^2g_j(z_j))\right]^2 \lambda_j^{2\delta} \left\lvert
\frac{\partial \bxi{j}(s)}{\partial x_k}\right\rvert^2 \nu(dz_j)
ds\\
\le C C_n \int_{\mathbb{R}_0} \left[\frac{1}{g(z)} \frac{d}{dz}
(z^2g(z))\right]^2 \nu(dz) \sum_{j=1}\int_0^t \lambda_j^{2\delta}
\left \lvert \frac{\partial \bxi{j}(s)}{\partial
x_k}\right\rvert^2 ds.
\end{align*}
Since, by Assumption \ref{Lev-Meas}-\eqref{LMi},
\begin{align*}
\left(\frac{1}{g(z)}\frac{d}{dz}[g(z)z^2]\right)^2=
\left(\frac{g^\prime(z)}{g(z)}z^2+2z\right)^2\\
\le 2 z^4 \left\lvert\frac{g^\prime(z)}{g(z)}\right\rvert^2+4
z^2\\
\le c \lvert z\rvert^2+ c\lvert z\rvert^4,
\end{align*}
and $\int_{\mathbb{R}_0} \lvert z\rvert^p\nu(dz)<\infty$ for any
$p\ge 2$ we infer that there exists $C>0$ such that
\begin{equation}\label{eq-6}
\sum_{k=1}^n \EE \lvert \ji{k}(t)\rvert^2 \le C C_n \sum_{j,k=1}^n
\EE\int_0^t \lambda_j^{2\delta}\biggl\vert \frac{\partial
\bxi{j}(s)}{\partial x_k}\biggr\vert^2 ds.
\end{equation}
Now, we treat the term involving $K^k(t)$. Owing to Assumption \ref{Lev-Meas}-\eqref{LMii} and \eqref{rem-kely} we can write
 \begin{align*}
 \ki{k}(t)=&-2 \sum_{j=1}^n \int_0^t \int_{\mathbb{R}_0} z_j
 \beta_j^{-1} \frac{\partial \bxi{j}(s)}{\partial x_k}
 d\tilde{\eta}_j(z_j,s) -2\sum_{j=1}^n \int_0^t\int_{\mathbb{R}_0} z_j
 \beta_j^{-1} \frac{\partial \bxi{j}(s)}{\partial x_k}
 \nu_j(dz_j)ds\\
 =&-2 \sum_{j=1}^n \int_0^t \int_{\mathbb{R}_0} z_j \beta_j^{-1}
 \frac{\partial \bxi{j}(s)}{\partial x_k} d\tilde{\eta}_j(z_j,s)
 -2\int_{\mathbb{R}_0} z\nu(dz) \sum_{j=1}^n \int_0^t \beta_j^{-1}
 \frac{\partial \bxi{j}(s)}{\partial x_k} ds\\
 =:&\ki{k}_1(t)+\ki{k}_2(t).
 \end{align*}
 Hence $$ \EE [ \ki{k}(t)]^2\le 2\EE \lvert \ki{k}_1(t)\rvert^2+ 2
 \EE \lvert \ki{k}_2(t) \rvert^2.$$ Let us first deal with $\EE
 \lvert \ki{k}_1(t)\rvert^2$. Using the discrete Cauchy-Schwarz
 inequality and It\^o's isometry we obtain
 \begin{align*}
 \EE \lvert \ki{k}_1(t)\rvert^2\le C \sum_{j=1}^n \beta_j^{-2}
 \lambda_j^{-2\delta} \times \sum_{j=1}^n \EE \biggl(\int_0^t
 \int_{\mathbb{R}_0} z_j \lambda_j^{2\delta} \frac{\partial
 \bxi{j}(s)}{\partial x_k} d\tilde{\eta_j}(z_j,s)\biggr)^2\\
 \le C C_n \sum_{j=1}^n \EE\biggl[\int_0^t\int_{\mathbb{R}_0} z_j^2
 \lambda_j^{2\delta} \left \lvert \frac{\partial
 \bxi{j}(s)}{\partial
 x_k}\right \rvert^2 \nu_j(dz_j) ds\biggr]\\
 \le C C_n \int_{\mathbb{R}_0} z^2 \nu(dz) \sum_{j=1}^n
 \EE\biggl[\int_0^t \lambda_j^{2\delta} \left \lvert \frac{\partial
 \bxi{j}(s)}{\partial x_k}\right \rvert^2 ds\biggr].
 \end{align*}
 Owing to the discrete Cauchy-Schwarz inequality we have
 \begin{align*}
 \EE \lvert \ki{k}_2(t) \rvert^2\le \biggl(\int_{\mathbb{R}_0}
  z\nu(dz)\biggr)^2 \sum_{j=1}^n\beta_j^{-2}
 \lambda_j^{-2\delta} \times \sum_{j=1} \EE\biggl(\int_0^t
 \lambda_j^\delta \frac{\partial \bxi{j}(s)}{\partial x_k}
 ds\biggr)^2\\
 \le \biggl(\int_{\mathbb{R}_0}  z \nu(dz)\biggr)^2
 \sum_{j=1}^n\beta_j^{-2} \lambda_j^{-2\delta} \times t
 \sum_{j=1}^n \EE\biggl[\int_0^t \lambda_j^{2\delta} \left\lvert
 \frac{\partial \bxi{j}(s)}{\partial x_k}\right\rvert^2 ds\biggr].
 \end{align*}
Hence there exists $C>0$ such that
\begin{equation}\label{eq-5}
\sum_{k=1}^n\EE [ \ki{k}(t)]^2\le C C_n (1+t)\sum_{j,k=1}^n
\EE\biggl[\int_0^t \lambda_j^{2\delta} \left\lvert \frac{\partial
\bxi{j}(s)}{\partial x_k}\right\rvert^2 ds\biggr].
\end{equation}
One can derive from \eqref{eq-5} and \eqref{eq-6} that
\begin{align*}
\sum_{k=1}^n \EE \lvert \ki{k}(t)\rvert^2 \le C C_n (1+t)
\EE\int_0^t
\lvert \mathrm{A}^\delta \nabla_x \bx(s)\rvert^2 ds,\\
\sum_{k=1}^n \EE \lvert \ji{k}(t)\rvert^2 \le C C_n \EE\int_0^t
\lvert \mathrm{A}^\delta \nabla_x \bx(s)\rvert^2 ds.
\end{align*}
Hence, by plugging these last identities in \eqref{COR-BEL} we
deduce that there exists $C$ such that
\begin{equation}
\lvert \nabla_x \EE [\Phi(\bx(t,x))]\rvert\le C C_n^\frac12
 \lvert \Phi\rvert_\infty \biggl[\EE \int_0^t
\lvert \mathrm{A}^\delta \nabla_x \bx(s,x) \rvert^2 ds
\biggr]^\frac12 (C^\frac12_2(t)(1+t)^\frac12 +C^\frac12_1(t))
\end{equation}
Now it remains to prove that $C_p(t)=\EE
\mathcal{A}(t)^{-2p}$ is finite. Since the measures
$\eta_j$, $j=1,2,\ldots$ are positive on $\mathbb{R}_0$ and $z_j^2>0$ we easily see
that
$$ \mathcal{A}(t)\ge \int_0^t\int_{\mathbb{R}_0} z_1^2 d\eta_1(z_1,s).$$
Thus, by Assumption \ref{Lev-Meas}-\eqref{LMiii} and \cite[Remark
3.2]{Takeuchi} it follows that
$$
\EE \mathcal{A}(t)^ {-2p} \le C\,\lk( t^ {-2p} + t^ {-\frac {4p}{\alpha}}\rk),
$$
and therefore,
$$ C_p(t)\le  C\,\lk( t^ {-2p} + t^ {-\frac {4p}{\alpha}}\rk)<\infty. $$
This ends the proof of our lemma.
\end{proof}
\section{Proof that $\burn$ converges to $\bu^R$ strongly in
$L^2(0,T;\h)$}\label{APP-CONV} In this section we are aiming to
prove that the Galerkin solution $\burn$ to \eqref{Hydro-2}
converges to the solution $\bu^R$ of \eqref{Hydro-1}. To do so we
consider the following system of finite dimensional differential
equations
\begin{equation}\label{Eq-7}
\begin{split}
\frac{d}{d t}\bv^R_n(t) +\mathrm{A}\bv^R_n(t)+\rho(\lvert
\bv^R_n(t)+\Pi_n \mathfrak{S}(t)\rvert^2/R) \Pi_n
\bb(\bv^R_n(t)+\Pi_n\mathfrak{S}(t),\bv^R_n(t)+\Pi_n\mathfrak{S}(t))
=0,\\
\bv^R_n(0)=\Pi_n \xi\in \h_n,
\end{split}
\end{equation}
where $\mathfrak{S}\in L^\infty(0,T;\h)$.
\begin{lem}\label{GAL-ST}
For any $\xi\in \h$, $\mathfrak{S}\in L^\infty(0,T;\h)$ there
exists $C>0$ such that
\begin{equation*}
\sup_{n\in \mathbb{N}, R>0} [\Lve \bv^R_n\Rve_{C(0,T;\h)}+\Lve
\bv^R_n\Rve_{L^2(0,T;\ve)}]<C,
\end{equation*}
\end{lem}
\begin{proof}
Multiplying \eqref{Eq-7} by $\bv^R_n$ yields that
\begin{equation*}
\begin{split}
\frac12 \frac{d}{dt} \lvert \bv^R_n(t)\rvert^2 +&\kappa \lvert
\mathrm{A}^\frac12 \bv^R_n(t)\rvert^2\\
=&-\langle \rho(\lvert \bv^R_n(t)+\Pi_n\mathfrak{S}(t)\rvert^2/R)
\Pi_n\bb(\bv^R_n(t)+\Pi_n
\mathfrak{S}(t),\bv^R_n(t)+\Pi_n\mathfrak{S}(t)),
\bv^R_n(t)\rangle.
\end{split}
\end{equation*}
Since $\rho$ is bounded by 1 and $\bb(\cdot,\cdot)$ is bilinear
\del{and $\Lve \Pi_n \Rve_{\mathcal{L}(\ve^\ast,\ve^\ast)}\le 1$}
we infer that
\begin{align*}
\lvert - \langle \rho(\lvert
\bv^R_n(t)+\Pi_n\mathfrak{S}(t)\rvert^2/R)
\Pi_n\bb(\bv^R_n(t)+\Pi_n
\mathfrak{S}(t),\bv^R_n(t)+\Pi_n\mathfrak{S}(t)),
\bv^R_n(t)\rangle\rvert\\
\le \lvert \langle \Pi_n \bb(\bv^R_n(t)+\Pi_n
\mathfrak{S}(t),\bv^R_n(t)+\Pi_n\mathfrak{S}(t)),
\bv^R_n(t)\rangle\rvert\\
\le \lvert  \langle \Pi_n
\bb(\bv^R_n(t)+\Pi_n\mathfrak{S}(t), \bv^R_n(t)+\Pi_n \mathfrak{S}(t)),
\bv^R_n(t)\rangle\lvert \\+ \lvert  \langle \Pi_n
\bb(\bv^R_n(t), \Pi_n\mathfrak{S}(t)),\bv^R_n(t)\rangle\rvert \\+ \lvert
 \langle
\Pi_n\bb(\Pi_n\mathfrak{S}(t),\Pi_n\mathfrak{S}(t)),\bv^R_n(t)\rangle\rvert.
\end{align*}
Now using Assumption \ref{B}-\eqref{B-a}, Assumption
\ref{B}-\eqref{B-c} and the fact $\Lve \Pi_n
\Rve_{\mathcal{L}(\ve^\ast,\ve^\ast)}\le 1$ we derive that
\begin{align*}
\lvert - \langle \rho(\lvert
\bv^R_n(t)+\Pi_n\mathfrak{S}(t)\rvert^2/R)
\Pi_n\bb(\bv^R_n(t)+\Pi_n
\mathfrak{S}(t),\bv^R_n(t)+\Pi_n\mathfrak{S}(t)),
\bv^R_n(t)\rangle\rvert\\ \le C\lve \bv^R_n(t)\rve \lve
\mathfrak{S}(t)\rve \lve \mathrm{A}^\frac12 \bv^R_n(t)\rve+ C \lve
\mathfrak{S}(t)\rve^2 \lve \mathrm{A}^\frac12 \bv^R_n(t)\rve.
\end{align*}
Hence, by the Young inequality we deduce that for any $\eps>0$ there exists a constant $C_\eps>0$ such that
\begin{align*}
\lvert - \langle \rho(\lvert
\bv^R_n(t)+\Pi_n\mathfrak{S}(t)\rvert^2/R)
\Pi_n\bb(\bv^R_n(t)+\Pi_n
\mathfrak{S}(t),\bv^R_n(t)+\Pi_n\mathfrak{S}(t)),
\bv^R_n(t)\rangle\rvert\\ \le C_\eps \lve \mathfrak{S}(t)\rve^2 \lve
\bv^R_n(t)\rve^2+C_\eps \lve \mathfrak{S}(t)\rve^4+\eps\lve
\mathrm{A}^\frac12 \bv^R_n(t)\rve^2.
\end{align*}
Therefore, by choosing $\eps=\kappa$ we infer that there exists a number
$C_\kappa>0$ such that
\begin{equation*}
\frac{d}{dt}\lve \bv^R_n(t)\rve^2 +\kappa \lve \mathrm{A}^\frac12
\bv_n^R(t)\rve^2 \le C_\kappa \lve \bv^R_n(t)\rve^2 \lve
\mathfrak{S}(t) \rve^2 +C_\kappa \lve \mathfrak{S}(t)\rve^4.
\end{equation*}
Owing to this last inequality and the Gronwall's inequality we
obtain
\begin{align*}
\sup_{s\in [0,T]} \lve \bv^R_n(s)\rve^2 \le (\lvert \xi\rvert^2
+C_\kappa \sup_{s\in [0,t]}\lve \mathfrak{S}(s)\rve^4)
e^{\sup_{s\in [0,T]}\lve \mathfrak{S}(s)\rve^2 T }=:C^\ast,\\
\int_0^T \lvert \mathrm{A}^\frac12 \bv^R_n(s)\rvert^2 ds\le \lve
\xi\rve^2 + C_\kappa \sup_{s\in [0,T]}\lve \mathfrak{S}(s)\rve^2
T(C^\ast+\sup_{s\in [0,T]}\lve \mathfrak{S}(s)\rve^2).,
\end{align*}
which implies the estimate in Lemma \ref{GAL-ST}.
\end{proof}
Thanks to Lemma \ref{GAL-ST} it is not difficult to prove that
there exists $C>0$ such that
\begin{equation*}
\sup_{n\in \mathbb{N}, R>0} \left\Lve \frac{d \bv^R_n}{d
t}\right\Rve_{\ve^\ast}\le C.
\end{equation*}
This estimate, the one in Lemma \ref{GAL-ST} and Aubin-Lions lemma
imply that there exists $\bv\in L^2(0,T;\h)$ and one can extract
subsequence, still denote by $\bv^R_n$, such that
\begin{align}\label{Eq-7-a}
\bv^R_n\rightarrow \bv \text{ strongly in } L^2(0,T;\h).
\end{align}
Also
\begin{align}
\bv^R_n\rightarrow \bv \text{ weakly-star in } L^\infty(0,T;\h),\label{Eq-7-b}\\
\bv^R_n\rightarrow \bv \text{ weakly in }
L^2(0,T;\ve).\label{Eq-7-c}
\end{align}
Now we can easily show that $\bv$ solves \eqref{Eq-8}. Now let
$\mathfrak{S}$ be the stochastic convolution defined in
\eqref{STOC-CONV-1}-\eqref{STOC-CONV-2}.
The stochastic process $\burn=\bv^R_n+\Pi_n\mathfrak{S}$ solves
\eqref{Hydro-2} and thanks to \eqref{Eq-7-a} we see that
\begin{equation}\label{CONV-GAL}
\burn \rightarrow \bu^R \text{ strongly in } L^2(0,T;\h) \,\,
\mathbb{P}-a.s.,
\end{equation}
where $\bu^R$ is the unique solution to \eqref{Hydro-1}.

\section{Acknowledgment}
We are very thankful to the Referee whose comments and remarks greatly improved the manuscript.
Razafimandimby's research is sponsored by the Austrian Science Fund (FWF) through the project M1487. Hakima Bessaih was supported in part by the Simons Foundation grant Nr. 283308,   the NSF grants DMS-1416689 and DMS-1418838.

\end{document}